\documentclass[lectures]{pcms-l}

 \usepackage{amssymb}

 \usepackage{graphicx}
\usepackage{epsfig}
\usepackage{verbatim}

  \theoremstyle{plain}
  \newtheorem{theorem}{Theorem} [lecture]
 \newtheorem{lemma}[theorem]{Lemma}
  \newtheorem{proposition}[theorem]{Proposition}

  \theoremstyle{definition}
 \newtheorem{definition}[theorem]{Definition}  
  \newtheorem{example}[theorem]{Example}
  \newtheorem{exercise}[theorem]{Exercise}

  \theoremstyle{remark}
 \newtheorem{remark}[theorem]{Remark}

  \numberwithin{equation}{chapter}

\newcommand{\Prob} {{\bf P}}
\newcommand{\Z}{{\mathbb Z}} 
\newcommand{\E}{{\bf E}}

\newcommand{\R}{{\mathbb{R}}}
\newcommand{\C}{{\mathbb C}}
\newcommand{\rad}{{\rm rad}}

\newcommand{\dist}{{\rm dist}}

\newcommand{\domains}  {{\mathcal D}}

\def \Im {{\rm Im}}
\def \Re {{\rm Re}}
\def \p {\partial}

\def \Half {{\mathbb H}}

\def \Disk {{\mathbb D}}

\def \hcap {{\rm hcap}}
\def \distsub {{\Upsilon}}

\def \F{{\mathcal F}}

\def \beq {\begin{equation}}
\def \eeq {\end{equation}}
\def \domains {{\mathcal D}}

\def \c {{\bf c}}
\def \O {{O}}
\def \bubble{{\Gamma}}
\def \rect {{\mathcal R}}

\def \tx {{\tilde \xi}}

\begin{document}


\frontmatter
\tableofcontents


\mainmatter


 \LectureSeries 
  {Schramm-Loewner Evolution ($SLE$)\author{G. Lawler}}



\address{Department of Mathematics\\University of Chicago\\
  5734 University Avenue\\Chicago, IL 60637-1546}
\email{lawler@math.uchicago.edu}
\thanks{Research supported by National Science Foundation
grant DMS-0405021.}






\section*{Introduction}

This is the first expository set of notes on SLE I have written since
publishing a book two years ago \cite{Lbook}.  That book covers material
from a year-long class, so I cannot cover everything there.  However,
these notes are not just a subset of those notes, because there is
a slight change of perspective.  The main differences are:
\begin{itemize}
\item  I have defined $SLE$ as a finite measure on paths that
is not necessarily a probability measure.  This seems more natural
from the perspective of limits of lattice systems and seems to
be more useful when extending $SLE$ to non-simply connected domains.
(However, I do not discuss non-simply connected domains in these
notes.)
\item  I have  made more use of
the Girsanov theorem in studying corresponding martingales
and local martingales.  
\end{itemize}
As in \cite{Lbook}, I will focus these notes on the continuous
process $SLE$ and will not prove any results about convergence
of discrete processes to $SLE$.  However, my first lecture will be
about discrete processes --- it is very hard to appreciate SLE
if one does not understand what it is trying to model.

I would like to thank Michael Kozdron,
Robert Masson, Hariharan Narayanan,
 and Xinghua Zheng for their assistance in the
preparation of these notes.  Figure 8 was produced by
Geoffrey Grimmett.

\lecture{Scaling limits of lattice models}

The Schramm-Loewner evolution (SLE) is a measure on continuous
curves that is a candidate for the scaling limit for discrete
planar models in statistical physics.  Although our lectures will
focus on the continuum model, it is hard to understand $SLE$ without
knowing some of
the discrete models that motivate it.  In this lecture, I will
introduce some of the discrete
models.  By assuming some kind
of ``conformal invariance'' in the limit, we will arrive at some
properties that we would like the continuum measure to satisfy.

\section{Self-avoiding walk (SAW)}  \label{sawsec}

A self-avoiding walk (SAW) of length $n$ in the integer lattice
$\Z^2 = \Z + i \Z$ is a sequence of lattice points
\[            \omega =[\omega_0,\ldots,\omega_n] \]
with $|\omega_j - \omega_{j-1}| = 1, j=1,\ldots,n$, and
$\omega_j \neq \omega_k$ for $j < k$.  If $J_n$ denotes
the number of SAWs of length $n$ with $\omega_0 = 0$, it is
well known that
\[               J_n \approx e^{\beta n} , \;\;\;\;
  n \rightarrow \infty ,\]
where $e^{\beta}$ is the {\em connective constant} whose value
is not known exactly. Here $\approx$ means that $\log J_n \sim
\beta n$ where $f(m) \sim g(m)$ means $f(m)/g(m) \rightarrow
1$.    In fact, it is believed that there
is an exponent, usually denoted $\gamma$, such that
\[                J_n \asymp n^{\gamma -1} \, e^{\beta n},
 \;\;\;\;
  n \rightarrow \infty, \]
where $\asymp$ means that each side is bounded by a constant times
the other.
The exponent $\nu$ is defined roughly by saying that the typical
diameter (with respect to the uniform probability measure on
SAWs of length $n$ with $\omega_0 = 0$) is of order $n^{\nu}$.
The constant $\beta$ is special to the square lattice, but the
exponents $\nu$ and $\gamma$ are examples of lattice-independent
{\em critical exponents}
that should be observable in a ``continuum limit''.  For example,
 we would expect the fractal dimension
of the paths in the continuum limit
to be $d= 1/\nu$.

To take a continuum limit we let $\delta > 0$ and
\[    \omega^{\delta}(j \delta^d ) = \delta \, \omega(j). \]
We can think of $\omega^\delta$ as a SAW on the lattice $\delta
\Z^2$ parametrized so that it goes a distance of order one in
time of order one.  We can use linear interpolation to make
$\omega^{\delta}(t)$ a continuous curve. 
  Consider the square in $\C$
\[            D= \{x+iy: -1 < x < 1 , -1 < y < 1 \} , \]
and let $z=-1, w = 1$.  For each integer $N$ we can consider a
finite measure on continuous curves $\gamma:(0,t_\gamma)
\rightarrow D$ with $\gamma(0+) = z, \gamma(t_\gamma) = w$
obtained as follows.  To each SAW $\omega$ of
length $n$ in $  \Z^2$
with $\omega_0 = -N, \omega_n = N$ and $\omega_1,\ldots,\omega_{n-1}
 \in ND$
we give measure $e^{-\beta n}$.  If we identify $\omega$ with
$\omega^{1/N}$ as above, this gives a measure on curves in $D$ from
$z$ to $w$.  The total mass of this measure is
\[           Z_N(D;z,w):=
      \sum_{\omega: Nz \rightarrow Nw,
    \omega \subset N D} e^{-\beta |\omega|}. \]
It is conjectured that there is a $b$ such that as $N \rightarrow \infty$,
\begin{equation}  \label{scalingrule}
            Z_N(D;z,w) \sim C(D;z,w) \, N^{-2b} . 
\end{equation}
Moreover, if we multiply by $N^{2b}$ and take a limit, then
there is a measure $\mu_D(z,w)$ of total mass $C(D;z,w)$ supported
on simple (non self-intersecting)
curves from $z$ to $w$ in $D$.  The dimension of these curves will
be $d = 1/\nu$.

\begin{figure}[htb]
\begin{center}
\epsfig{file=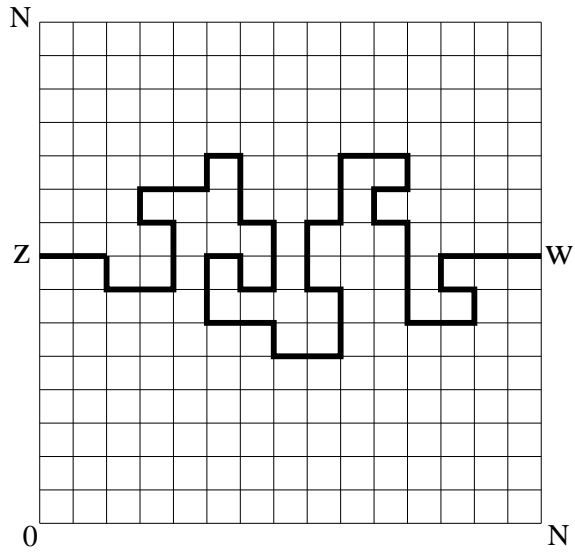}
\caption{Self-avoiding walk in a domain} \label{SAW}
\end{center}
\end{figure}

\begin{figure}[htb]
\begin{center}
\epsfig{file=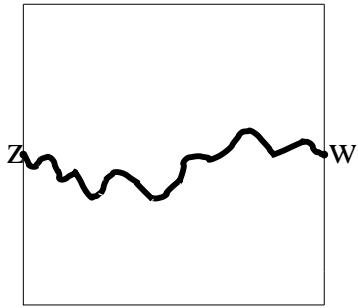}
\caption{Scaling limit of SAW}
\end{center}
\end{figure}

Similarly, if $D$ is another domain and $z,w \in \p D$, we can
consider SAWs from $z$ to $w$ in $D$.  If $\p D$ is smooth
at $z,w$, then (after taking care of the local lattice effects ---
we will not worry about this here), we define the measure
as above, multiply by $N^{2b}$ and take a limit.  We conjecture
that we get a measure $\mu_D(z,w)$ on simple curves from $z$
to $w$ in $D$.  
We write the measure $\mu_D(z,w)$ as
\[           \mu_D(z,w) = C(D;z,w) \, \mu_D^\#(z,w) , \]
where $\mu_D^\#(z,w)$ denotes a probability measure.

\begin{figure}[htb]
\begin{center}
\epsfig{file=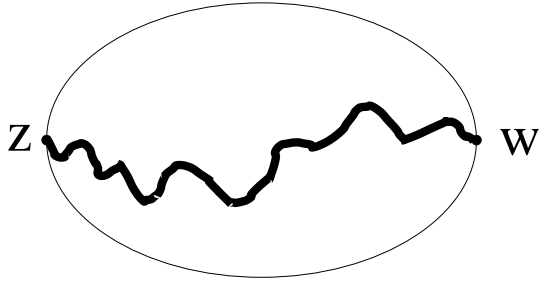}
\caption{Scaling limit of SAW in a different domain}
\end{center}
\end{figure}

It is believed that the scaling limit satisfies some
kind of ``conformal invariance''.  To be more precise
we assume the following 
  conformal {\em covariance} property: if
$f: D \rightarrow f(D)$ is a conformal transformation
and $f$ is differentiable in neighborhoods of $z,w \in \p D$,
then
\[        f \circ \mu_D(z,w) = |f'(z)|^b \, |f'(w)|^b \,
    \mu_{f(D)}(f(z),f(w)). \]
In other words the total mass satisfies the scaling
rule
\[       C(D;z,w) = |f'(z)|^b \, |f'(w)|^b \, C(f(D);f(z),f(w)) , \]
and the corresponding probability measures are conformally
{\em invariant}:
 \[     f \circ \mu_D^\#(z,w) =  
    \mu_{f(D)}^\#(f(z),f(w)). \]

Let us be a little more precise about the definition of
$  f \circ \mu_D^\#(z,w)$.  Suppose $\gamma:(0,t_\gamma) \rightarrow
D$ is a curve with $\gamma(0+) = z, \gamma(t_\gamma-) = w$.
For ease, let us assume that $\gamma$ is simple.  Then the curve
$f \circ \gamma$ is the corresponding curve from $f(z)$ to $f(w)$.
At the moment, we have not specified the parametrization of
$f \circ \gamma$.  We will consider two possibilities:
\begin{itemize}
\item  {\bf Ignore the parametrization}.  We consider two curves
equivalent if one is an (increasing)
 reparametrization of the other.  In this case we do not need
to specify how we parametrize $f \circ \gamma$.
\item {\bf Scaling by the dimension $d$.}  If $\gamma$ has
the parametrization as given in the limit, then the amount
of time need for $f \circ \gamma$ to traverse $f(\gamma[t_1,t_2])$
is
\[      \int_{t_1}^{t_2} |f'(\gamma(s))|^d \, ds . \]
\end{itemize}
In either case, if we start with the probability
measure $\mu_D^\#(z,w)$,
the transformation $\gamma \mapsto f\circ \gamma$ induces a
probability measure which we call $f \circ \mu_D^\#(z,w)$.

\begin{remark} $\;$

\begin{itemize}
\item Since $\mu_D^\#(z,w)$ is a conformal {\em invariant}, we can
define $f  \circ \mu_D^\#(z,w)$ even if $\p D$ is not
smooth at $z,w$.  (For really bad conformal tranformations, one
needs to worry about the continuity of $f \circ \gamma$ at 
$f(z)$ and $f(w)$, but we do not need to consider transformations
that are {\em that} bad!)
\item The Riemann mapping theorem tells us that if $D,D_1$
are simply connected domains; $z,w$ distinct points in $\p D$;
$z_1,w_1$ distinct points in $\p D_1$, then there is a 
one-parameter family of conformal transformations $f:D \rightarrow
D_1$ with $f(z) = z_1, f(w)= w_1$.  In particular, if we know the measure
for one simply connected domain $D$, we know it for all 
simply connected domains.
\item  In particular, if we know $\mu_\Half^\#(0,\infty)$, then
we know $\mu_D(z,w)$ for all simply connected domains.  Here
$\Half$ denotes the upper half plane.
\item The measure $\mu_\Half^\#(0,\infty)$ must be invariant
under the dilation $z \mapsto rz$ ($r>0$).
\item  Although the map $f:D \rightarrow D_1$ with $f(z) = z_1,
f(w) = w_1$ is not uniquely defined the quantity $f'(z) \, f'(w)$
is independent of the choice.  
\item One can choose a unique
such $f$ with $|f'(w)| = 1$ in which case the conformal covariance
condition becomes  
 \[ f \circ \mu_D(z,w) = |f'(z)|^b \,  
    \mu_{f(D)}(f(z),f(w)). \]
When $D$ is a subdomain of $\Half$ with $\Half \setminus D$
bounded and $w = f(w) = \infty$, then the condition $|f'(w)| = 1$
translates to $f(w') \sim w'$ as $w' \rightarrow \infty$.
\end{itemize}
\end{remark}

There are two more properties that we would expect the family
of measures $\mu_D(z,w)$ to have.  The first of these will
be shared by all the examples in this section while the second
will not.  We just state the properties, and leave it to the
reader to see why one would expect them in the limit.

\begin{itemize}

\item  {\bf Domain Markov property}.  Consider the measure
$\mu^\#_D(z,w)$ and suppose an initial segment of
the curve $\gamma(0,t]$ is observed.  Then the conditional
distribution of the remainder of the curve given $\gamma(0,t]$
 is the same
as $\mu_{D \setminus \gamma(0,t]}^\#(\gamma(t),w).$

\begin{figure}[htb]
\begin{center}
\epsfig{file=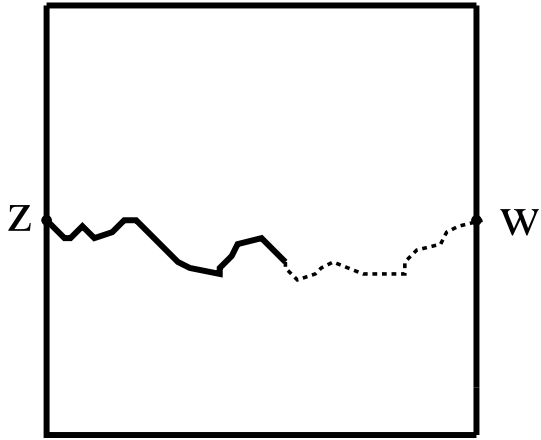}
\caption{Domain Markov property}
\end{center}
\end{figure}

\item  {\bf Restriction property}.  Suppose $D_1 \subset D$.
Then $\mu_{D_1}(z,w)$ is $\mu_D(z,w)$ restricted to paths
that lie in $D_1$.   In terms of Radon-Nikodym derivatives, this
can be phrased as
\[            \frac{d \mu_{D_1}(z,w)}{  d \mu_D(z,w)} (\gamma)
   = 1\{\gamma(0,t_\gamma) \subset D_1\} . \]

\begin{figure}[htb]
\begin{center}
\epsfig{file=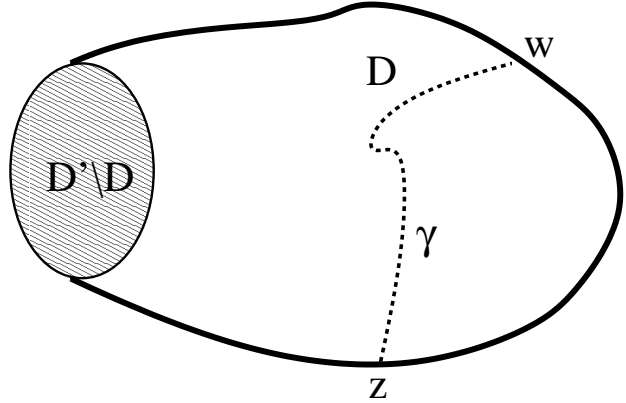}  \label{boundper}
\caption{Illustrating the restriction property}
\end{center}
\end{figure}

\end{itemize}

We have considered the case where $z,w \in \p D$.  We could
consider $z \in \p D, w \in D$.
In this case the measure is defined similarly, but \eqref{scalingrule}
becomes
\[   Z_D(z,w) \sim  C(D;z,w) \, N^{-b} \, N^{- \tilde b}, \]
where $\tilde b$ is a different exponent (see
Lectures 5 and
6).  The limiting
measure $\mu_D(z,w)$ would satisfy the conformal covariance
rule
\[          f \circ \mu_D(z,w) = |f'(z)|^b \, |f'(w)|^{\tilde b}
  \, \mu_{f(D)}(f(z),f(w)). \]
Similarly we could consider $\mu_D(z,w)$ for $z,w \in D$.

\section{Loop-erased random walk}  \label{lerwsec}

We start with simple random walk.  Let $\omega$ denote a nearest
neighbor random walk from $z$ to $w$ in $D$.  We no longer put
in a self-avoidance constraint.  We give each walk $\omega$
measure $4^{-|\omega|}$ which is the probability that the
first $n$ steps of an ordinary
random walk in $\Z^2$ starting at $z$ are $\omega$. The total
mass of this measure is the probability that a simple random
walk starting at $z$ immediately goes into the domain and
then leaves the domain at $w$.  Using the ``gambler's ruin'' estimate
for one-dimensional random walk, one can show that the total
mass of this measure decays like $O(N^{-2})$;  in
fact (Exercise \ref{may2.exer1})
\begin{equation}  \label{may2.1}
 Z_N(D;z,w) \sim C(D;z,w) \, N^{-2}, \;\;\;\; N \rightarrow
\infty,
\end{equation}
where $C(D;z,w)$ is  the ``excursion Poisson kernel'',
$H_{\p D}(z,w)$,
defined to be the normal derivative of the
Poisson kernel $H_D(\cdot,w)$ at $z$.
  In the notation of the previous section
$b=1$.  For each realization of the walk, we produce a self-avoiding
path by erasing the loops in chronological order.

\begin{figure}[htb]  \label{simplefigure}
\begin{center}
\epsfig{file=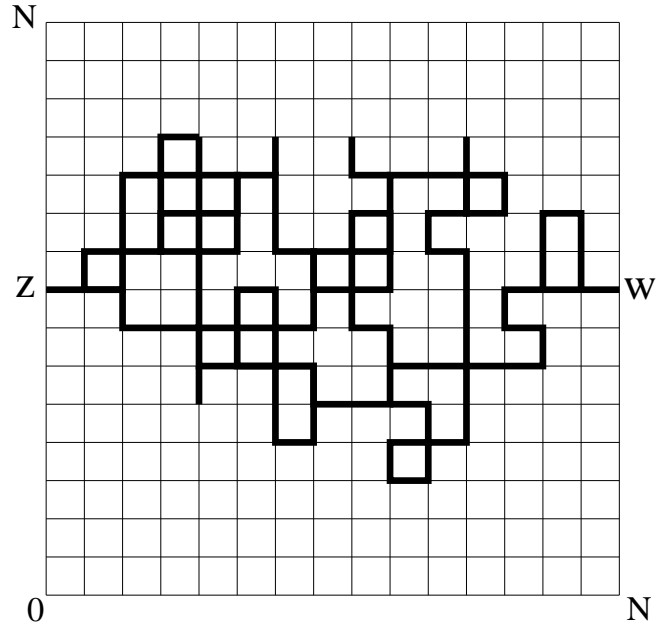}
\caption{Simple random walk in $D$}
\end{center}
\end{figure}

\begin{figure}[htb]
\begin{center}
\epsfig{file=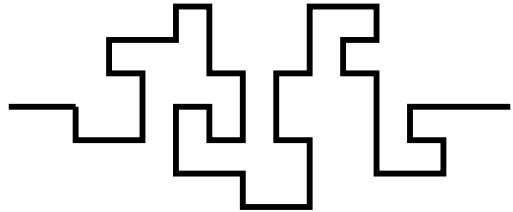}
\caption{The walk obtained from erasing loops
 chronologically}
\end{center}
\end{figure}

Again we are looking for a continuum limit $\mu_D(z,w)$
with paths of
dimension $d$ (not  the same $d$ as for SAW).  The limit
should satisfy
\begin{itemize}
\item Conformal covariance
\item Domain Markov property
\end{itemize}
However, we would not expect the limit to satisfy the restriction
property.  The reason is that the measure given to each self-avoiding
walk $\omega$ by this procedure
is determined by the number of ordinary random walks which
produce $\omega$ after loop erasure.  If we make the domain smaller,
then we lose some random walks that would produce $\omega$ and hence
the measure would be smaller.  In terms of Radon-Nikodym derivatives,
we would expect
\[             \frac{d \mu_{D_1}(z,w)}{  d \mu_D(z,w)} < 1 . \]

\section{Percolation}  \label{percsec}

Suppose that every point in the triangular lattice in the upper
half plane is colored black or white independently with each
color having probability $1/2$.  A typical realization is
 illustrated in Figure
8 
(if one ignores the bottom row).

\begin{figure}[htb]
\begin{center}
\includegraphics[width=10cm]{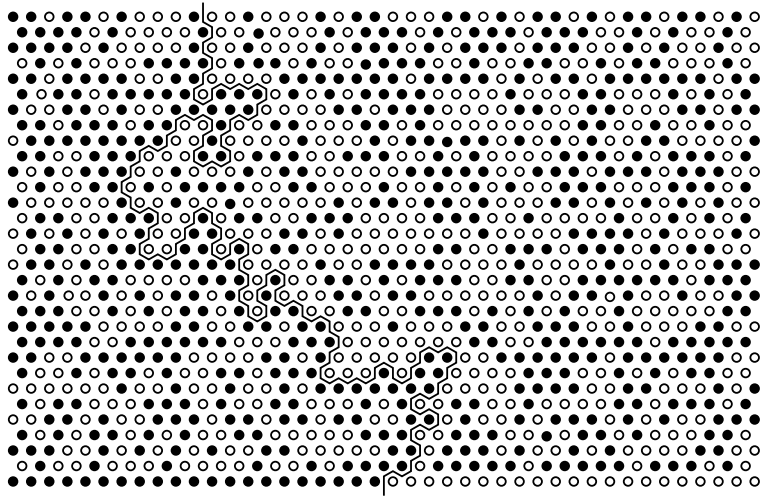}
\caption{The percolation exploration process.}
\end{center}
\label{beffarafig2}
\end{figure}

We now put a boundary condition on the bottom row as illustrated ---
all black on one side of the origin and all white on the other side.
For any realization of the coloring, there is a unique
curve starting at the bottom row that has all white vertices on
one side and all black vertices on the other side.  This is called
the {\em percolation exploration process}.  Similarly we could start
with a domain $D$ and two boundary points $z,w$; give a boundary
condition of black on one of the arcs and white on the other arc;
put a fine triangular lattice inside $D$; color vertices black or
white independently with probability $1/2$ for each; and then consider
the path connecting $z$ and $w$.  In the limit, one might hope
for a continuous interface.  In comparison to the previous examples,
the total mass of the lattice measures is one; another way of saying
this is that $b = 0$.  We suppose that the curve is conformally
invariant, and one can check that it should satisfy the domain
Markov property.

The scaling limit of percolation satisfies another property called
the {\em locality property}.  Suppose $D_1 \subset D$ and $z,w
\in \p D \cap \p D_1$ as in Figure 5.  Suppose that only an initial
segment of $\gamma$ is seen.  To determine the measure of the
initial segment, one only observes the value of the percolation
cluster at vertices adjoining $\gamma$.  Hence the measure of
the path is the same whether it is considered as a curve in
$D_1$ or a curve in $D$.  The locality property is stronger
than the restriction property which SAW satisfies.  The restriction
property is a similar statement that holds for the
entire curve $\gamma$ but not for all initial segments of
$\gamma$.

\begin{figure}[htb]
\begin{center}
\epsfig{file=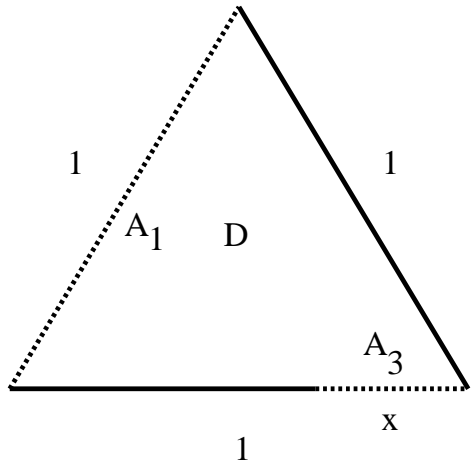}
\caption{Cardy's formula:
  $P_D(A_1,A_3) = x$.}
\end{center}
\end{figure}

There is another well known conformal invariant for percolation
known as Cardy's formula, named after the physicist who first
predicted\footnote{The word ``predicted'' here means that the formula
was found from a nontrivial, but not mathematically rigorous,
argument. Much of the work by theoretical physicists using
conformal field theory falls into this category.  Although nonrigorous,
the ideas are deep and involve a number of different areas
of mathematics.}  the formula.  Suppose $D$ is a simply
connected domain and the boundary is divided into four arcs,
$A_1,A_2,A_3,A_4$ in counterclockwise order.  Let $P_D(A_1,A_3)$
be the limit as the lattice spacing goes to zero of the probability
that in a percolation cluster as above there is a connected
cluster of white vertices connecting $A_1$ to $A_3$.  This
should be a conformal invariant.  It turns out that the nicest
domain to give the formula is an equilateral triangle as shown in
the Figure 9.

\section{Ising model}

The Ising model is a simple model of a ferromagnet.  We will consider
the triangular lattice as in the previous section.  Again we color
the vertices black or white although we now think of the colors as
spins.  If $x$ is a vertex, we let $\sigma(x)  = 1$ if $x$
is colored black and $\sigma(x) = -1$ if $x$ is colored white.  The
measure on configurations is such that neighboring spins
like to have the same sign.

\begin{figure}[htb]
\begin{center}
\epsfig{file=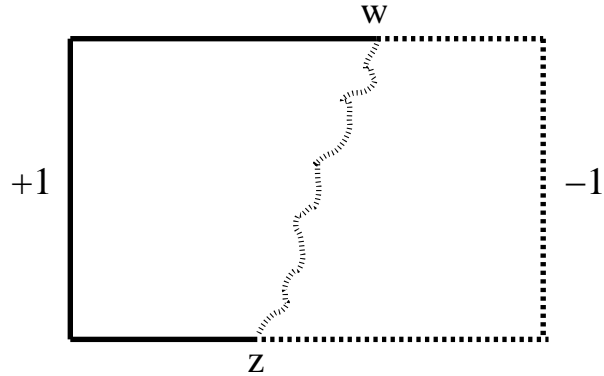}
\caption{Ising interface.}
\end{center}
\end{figure}

It is easiest to define the measure for
a finite collection of spins.  Suppose $D$ is a bounded domain in
$\C$ with two marked boundary points $z,w$ which give us two boundary
arcs.  We consider a fine triangular lattice in $D$; and fix boundary
conditions $+1$ and $-1$ respectively on the two boundary arcs.  Each
configuration of spins is given energy
\[                    {\mathcal E} = -\sum_{x \sim y}
               \sigma(x) \, \sigma(y) , \]
where $x \sim y$ means that $x,y$ are nearest neighbors.  We then
give measure $e^{-\beta \mathcal E}$ to a configuration of spins.
If $\beta$ is small, then the correlations are localized and spins
separated by a large distance are almost independent.  If $\beta$ is 
large, there is long-range correlation.  There is a critical
$\beta_c$ that separates these two phases.

For each configuration of spins there is a well-defined boundary
between $+1$ spins and $-1$ spins defined in exactly the same way as
the percolation exploration process.  At the critical $\beta_c$
it is believed that this gives an interesting fractal curve and
that it should satisfy conformal covariance and the domain Markov
property.

\section{Assumptions on limits}

Our goal is to understand the possible continuum limits for these
discrete models.  We will discuss the boundary to boundary case here
but one can also have boundary to interior or interior to interior.
(The terms ``surface'' and ``bulk'' are often used for boundary
and interior.)  Such a limit is a measure $\mu_D(z,w)$ on curves
from $z$ to $w$ in $D$ which can be written
\[               \mu_D(z,w) = C(D;z,w) \, \mu_D^\#(z,w), \]
where $\mu_D^\#(z,w)$ is a probability measure.  The
existence of 
$\mu_D(z,w)$ assumes smoothness of $\p D$ near $z,w$, but the
probability measure $\mu_D^\#(z,w)$ exists even without the smoothness
assumption.  The two basic assumptions are:
\begin{itemize}
\item  Conformal covariance of $\mu_D(z,w)$ and
  conformal invariance  of  $\mu_D^\#(z,w)$.
\item  Domain Markov property.
\end{itemize} 

The starting point for the Schramm-Loewner evolution is to show
that {\em if we ignore the parametrization of the curves}, then
there is only a one-parameter family of probability measures
$\mu_D^\#(z,w)$ for {\em simply connected} domains $D$ that
satisfy conformal invariance and the domain Markov property.
We will construct this family. The parameter is usually denoted
$\kappa > 0$.  By the Riemann mapping theorem, it
suffices to construct the measure for one simply connected domain
and the easiest is the upper half plane $\Half$ with boundary
points $0$ and $\infty$.  As we will see, there are a number
of ways of parametrizing these curves.

\section{Exercises for Lecture 1}

\begin{exercise}
let $J_n$ denote the number of SAWs of length $n$ with $\omega_0=0$
in $\Z^2$.
\begin{enumerate}
\item  Show that there exists a $\beta$ with $2 \leq
e^\beta \leq 3$  such that
\[         \lim_{n \rightarrow \infty} \frac{\log J_n}{n} = \beta . \]
\item Prove that $2 < e^{\beta} < 3$. 
\end{enumerate}
\end{exercise}

\begin{exercise}  Suppose $D$ is a simply connected
domain and $f:D \rightarrow D$ is a conformal
transformation and $z,w \in \p D$.  Suppose $\p D$ is smooth
near $z,w$ and $f(z) = z, f(w) = w$.  Show that $f'(z) \, f'(w) = 1$.
\end{exercise}

\begin{exercise} \label{may2.exer1}
 Let 
\[  V =  V_N = \{ j + ik \in \Z \times i\Z  : 1 \leq
  j \leq N-1, 1 \leq k \leq N-1\}. \]
Let $S_n$ denote a simple random walk starting at
$j + i k \in V_N$ and let $\tau = \tau_N = \min\{n:
S_n \not\in V_N\}. $  Let
\[       H(j+ik,l +im ) = \Prob^{j+ik} \{S_\tau = l+im \}. \]
We will compute this.  Let $\p V  = \{z \in \Z^2
: \dist(z,V) = 1\}$ and $\overline V = V \cup
\p V$.   Let ${\mathcal A}$ denote the set of real-valued
functions $f$ on 
$\overline V$ satisfying:
\begin{itemize}
\item  $f$ is discrete harmonic in $V$ (i.e., the value of $f$
is the average of $f$ at its nearest neighbors.
\item  $f \equiv 0$  on $\p V \setminus \{N + im:m=1,\ldots,N-1\}.$
\end{itemize}
Show the following:
\begin{enumerate}
\item  Show that there exist $N-1$ linearly independent functions
$f_1,\ldots,f_{N-1} \in {\mathcal
A}$ of the form
\[           f_q(j+ik) = \sinh(r_{q} j) \, \sin(s_q k),
\;\;\;\;\; q=1,\ldots,N-1 . \]
\item Show that every  $f \in {\mathcal A}$ is a linear combination
of $f_1,\ldots,f_{N-1}$. 
\item For any $m=1,\ldots,N-1$, find the unique $f \in {\mathcal A}$
for which
\[          f(N+im) = 1, \;\;\;\;\;\; f(N + im') = 0, m' \neq m. \]
\end{enumerate}
Use this to justify \eqref{may2.1} in Section \ref{lerwsec}
and show that for non-corner points $z$,
\[        C(D;z,w) = H_{\p D}(z,w) := \frac{d}{dn}
               H_D(z,w) , \]
where $d/dn$ denotes normal derivative and $H_D$
is the Poisson kernel.
\end{exercise}

 \lecture{Conformal mapping and Loewner equation}

\section{Important results about conformal maps}

Here we  summarize the basic facts
about conformal maps that one needs in order to use the Loewner
equation effectively.  A {\em domain} $D$ will be a connected subset
of $\C$.  We will call a holomorphic function
\[               f:D \longrightarrow D' \]
a {\em conformal transformation} if it is one-to-one and onto.  This
implies that $f'(z) \neq 0$ for all $z \in D$.  A domain
$D$ is {\em simply connected} if $\C \setminus D$ is connected. 
We let $\Disk = \{|z|<1\}$ denote the unit disk and $\Half =\{x+iy:
y > 0\}$ the upper half plane.  The
starting point is the Riemann mapping theorem.

\begin{theorem}[Riemann mapping theorem]  If $D$ is a proper,
simply connected domain in $\C$ and $z \in D$, then there is a
unique conformal transformation 
\[              f: \Disk \rightarrow D \]
with $f(0) = z, f'(0)  > 0$.
\end{theorem}

In other words, there is a one-to-one correspondence between
one-to-one, analytic functions $f$ on $\Disk$ with $f(0) = 0,
f'(0) > 0$ and simply connected proper subdomains of $\C$
containing the origin.  The  class of such functions
with $f'(0) = 1$.  is denoted
${\mathcal S}$.  Classical function theory devoted much time
to the study of ${\mathcal S}$.  The high point was the proof
by de Branges of the Bieberbach conjecture.

\begin{theorem} [Bieberbach conjecture, de Branges theorem]
If $f \in {\mathcal S}$ has power series expansion
\[               f(z) =  z + \sum_{n=2}^\infty a_n \, z^n , \]
then $|a_n| \leq n$ for each $n$.
\end{theorem}

Although we do not need such a deep result, we do use some
facts that were developed to try to solve the conjecture.  

\begin{theorem}  [Koebe $(1/4)$-theorem]
If $f \in {\mathcal S},$  then $(1/4) \, \Disk \subset
f(\Disk)$.
\end{theorem}

\begin{theorem} [Distortion theorem]   If $f \in {\mathcal S}$,
and $|z| \leq r < 1$, then
\begin{equation}  \label{apr17.1}
  \frac{1-r}{(1+r)^3} \leq |f'(z)| \leq \frac{1+r}{(1-r)^3}. 
\end{equation}
\end{theorem}

The particular values $1/4$ and those in \eqref{apr17.1} are
not as important as the fact that there is some uniform bound
over all $f \in {\mathcal S}$.  These theorems are important
for studying conformal maps even when the domains are not
simply connected.  Suppose $f: D_1 \rightarrow D_2$ is
a conformal transformation with $f(0) = 0$.  Let
$d_j = \dist(0,\p D_j)$.  If 
\[                \tilde f(z) = \frac{f(d_1z)}{d_1 \, f'(0)}, \]
then $\tilde f \in {\mathcal S}$.  Therefore, $(1/4)
\, \Disk \subset \tilde f(\Disk)$ which implies $d_2
 \geq |f'(0)| \, d_1/4$.  By interchanging the roles of
$D_1,D_2$, we get the corollary
\[                  \frac{|f'(0)|}{4} \leq \frac{d_2}{d_1}
    \leq 4 \, |f'(0)| . \]

There is a similar result about harmonic functions that is
simple but worth emphasizing.  We will state it for the gradient
but there are similar results for higher derivatives. These
are just corollaries of the Poisson integral formula,
\[          u(z) = \int_{\p \Disk} u(w) \, H_\Disk(z,w) \, |dw|, \]
where 
\[                H_\Disk(z,w) =  \frac{1}{2\pi} \, \frac{1-|z|^2}
       {|z-w|^2}  \]
is the Poisson kernel in $\Disk$
and $|dw|$ means integration with respect
to arc length.

\begin{proposition}  For every $r < 1$, there exists $c_r <
\infty$ such that the following holds for all $|z| \leq r$.

\begin{itemize}

\item  (Harnack inequality)
If  $u: \Disk \rightarrow (0,\infty)$
is harmonic, then
\[          c_r^{-1} \, u(0) \leq u(z) \leq c_r \, u(0).\]

\item (Derivative estimates)
If  $u: \Disk \rightarrow \R$ is harmonic, then
\[          |\nabla u(z)| \leq c_r \, \|u\|_\infty. \]
   
\end{itemize}
In particular, there is a $c$ such that
if $u: D \rightarrow \R$ is harmonic, then
\[              |\nabla u(z)| \leq   \frac{c}{\dist(z,\p D)}
  \,  \|u\|_\infty. \]
\end{proposition}
 
The next theorem is a corollary of a stronger theorem known
as the Beurling projection theorem.  However, the weaker version
here is what is used most often in applications (and also has
discrete analogues).  

\begin{theorem}  [Beurling estimate] There is a $c < \infty$ such
that the following is true.
 Suppose $\gamma:[0,1] \rightarrow
\C$ is a continuous curve with $|\gamma(0)| = r < 1 = |\gamma(1)|$.
Let $B_t$ be a complex Brownian motion starting at the origin
and let $\tau = \tau_\Disk = \inf\{t: |B_t| = 1\} . $  Then,
\[  \Prob\{B[0,\tau] \cap \gamma[0,1] = \emptyset \}
           \leq c \, r^{1/2} . \]
\end{theorem}

\begin{remark}  The estimate is sharp when $\gamma$ is a line
segment from $r$ to $1$.  See Exercise \ref{may3.exer2}.
\end{remark}

\section{Half-plane capacity}   \label{hcapsec}

If $K$  is a bounded, relatively closed
subset of $\Half$, let $D = D_K = \Half \setminus K$ and
\[  \phi_D(z) = \Im(z) - \E^z[\Im(B_{\tau_D})] , \]
where $B_t$ is a standard complex Brownian motion and $\tau_D
= \inf\{t: B_t \not \in D\}.$  Then
$\phi_D$ is a positive harmonic function on $D$ that
vanishes on $\p D$ \footnote{Actually we can only assert that
it vanishes at the
{\em regular} points of the boundary.  We will not define
regular here, but if all the connected components
of $\p D$ are larger than singletons then all points on $\p D$
are regular.} and such that
\[   \phi_D(z) = \Im(z) + O(|z|^{-1}) , \;\;\;\;  z \rightarrow
\infty . \]
The {\em half-plane capacity (from infinity)} of $K$ is
defined by
\[        \phi_D(z) = \Im\left(z  +  \frac {\hcap(K)} z\right )
  + o(1) , \;\;\;\;
  z \rightarrow \infty , \]
or, in other words,
\[   \hcap(K) = \lim_{z \rightarrow \infty} 
                       \Im(-1/z) \, \E^z[\Im(B_{\tau_D})] . \]
The existence of the limit is included in the proof
of the following
lemma. Let
$\Disk_+ = \Disk \cap \Half$ denote the upper half disk.

\begin{lemma}
$\;$

\begin{itemize}

\item
If $r > 0$,
\[   \hcap(rK) = r^2\, \hcap(K), \;\;\;\;\;
  \hcap(r + K) = \hcap(K) . \]

\item
If $K \subset \overline {\Disk_+}$, then
\begin{equation}  \label{may3.1}
  \hcap(K) = \int_0^\pi \E^{e^{i\theta}}[\Im(B_{\tau_{D}})]
  \, \left(\frac 2\pi \, \sin \theta\right) \, d\theta. 
\end{equation}

\end{itemize}
\end{lemma}

\begin{proof} (sketch) 
The scaling rule follows immediately from the scaling rule
$\phi_D(z) = \phi_{rD}(rz)/r$, and the translation invariance
by the rule $\phi_{D+r}(r+z) = \phi_D(z)$.  
 The last equality follows by taking a
Brownian motion starting at $z$ and considering the hitting
distribution of $\R \cup \overline{\Disk_+}$, restricted
to the unit circle.  Then it can be shown  
(Exercise \ref{may3.exer1})
 that as $z \rightarrow
\infty$, the hitting density of the unit circle is given by
\begin{equation}  \label{may3.3}
    \frac{2}{\pi}  \,  \Im(-1/z)\, \sin \theta  \, [1 + O(|z|^{-1})],
\;\;\;\;  z\rightarrow \infty . 
\end{equation}
In other words,
 \[\Im(-1/z)^{-1}\,  \E^z[\Im(B_{\tau_D})] =
     \left[\int_0^\pi \E^{e^{i\theta}}[\Im(B_{\tau_{D}})]
  \, \left(\frac 2\pi \, \sin \theta\right) \, d\theta\right]
 \, [1 + O(|z|^{-1})]. \]

\end{proof}

\begin{remark} From the lemma one can conclude
immediately that $\hcap(\overline {\Disk_+})
 = 1$.   One can also see this by noting that
the function $f(z) = z + z^{-1}$ maps
$\Half \setminus \overline {\Disk_+}$ conformally
onto $\Half$.  If $K_1 \subset K$, then $\hcap(K_1)
\leq \hcap(K)$.  In particular, we get the
estimate
\[             \hcap(K) \leq \rad(K)^2 , \]
where $ \rad(K) = \sup\{|z|:z \in K\}$.  There is
no corresponding bound in the opposite direction even
for simply connected $D$.  In fact, one can check
that there is a $c$ such that for all $K$
\[      
 \hcap(K) \leq c \,\rad(K) \,  \sup\{\Im(z) ; z \in K\}.\]
\end{remark}

\begin{remark}
The proof of \eqref{may3.1}
also gives an error estimate. There is
a $c < \infty$ such that for all  $K \subset \overline
{\Disk_+}$, and $|z| \geq 2$,
\begin{equation} \label{may3.4}
      \left| \,  \Im(-1/z)^{-1} \, \E^z[\Im(B_{\tau_{D_K}})]
                 - \hcap(K)\, \right| \leq \frac{c \, \hcap(K)}
  {|z|}. 
\end{equation}
By scaling, this implies that for any $K$ and any $|z|
\geq 2 \rad(K)$,
\begin{equation} \label{oct17.1}
      \left| \,  \Im(-1/z)^{-1} \, \E^z[\Im(B_{\tau_{D_K}})]
                 - \hcap(K)\, \right| \leq \frac{c \, \hcap(K)
  \, \rad(K)}{|z|}. 
\end{equation}
Note that the error is of order $\hcap(K) \, \rad(K)$ rather than
$\hcap(K)^2$.  As mentioned above, $\rad(K)$ can be much larger
than $\hcap(K)$.
\end{remark}
 
\begin{remark}
In much of the literature on SLE, the half-plane capacity is called
just the capacity and denoted ${\rm cap}$.  However, this can lead
to confusion because there are other natural definitions of capacities
of sets in $\Half$.
\end{remark}

If $D = \Half \setminus
K$ is simply connected, then $\phi_D$ is the imaginary part
of a conformal transformation $g_D: D \rightarrow \Half$
\[     g_D  = \Re[g_D] + i \phi_D . \]
We will also write this as $g_K$. 
This defines $\Re[g_D]$ up to an additive constant.  We define
$g_D$ uniquely by specifying that the additive constant should
be ``0 at infinity'', i.e., so that $g_D$ has the expansion
\[             g_D(z) = z + \frac{\hcap(K)}{z}
    + O(|z|^{-2}) , \;\;\;\;\;  z \rightarrow \infty . \]
There is a error estimate similar to (\ref{oct17.1}),
\begin{equation}  \label{oct17.2}
   \left|g_D(z) - \left( z - \frac{\hcap(K)}{z}
  \right) \right| \leq 
\frac{c \, \hcap(K)
  \, \rad(K)}{|z|^2}, \;\;\;\;  |z| \geq 2 \rad(K). 
\end{equation}

\begin{figure}[htb]
\begin{center}
\epsfig{file=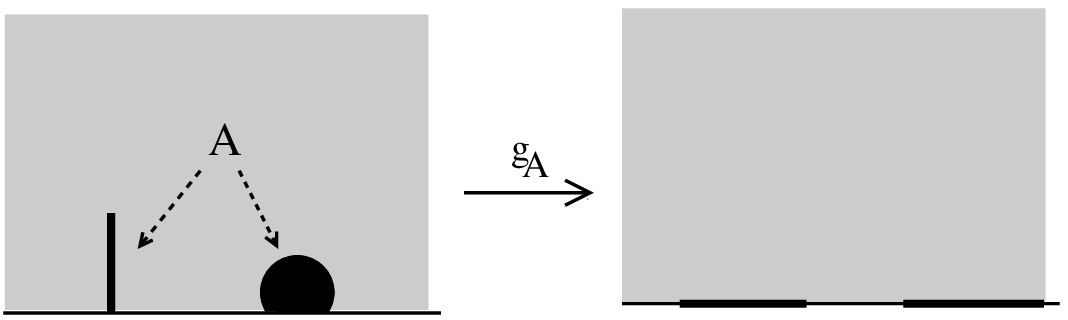}
\caption{The conformal transformation $g_A = g_{\Half \setminus A}$.}
\end{center}
\end{figure}

\section{Loewner equation} The last inequality implies
  a version of the
Loewner equation.  

\begin{proposition}
Suppose for each $t > 0$ there is a set $K_t$
as above.  Suppose $\dot a(0) = \p_t[\hcap(K_t)]|_{t=0+}$
exists and $r_t:= \rad(K_t) \rightarrow 0$ as $t \rightarrow 0+$.
Let $\phi_t = \phi_{K_t}$.  Then for fixed $z \in \Half$,
as $t \rightarrow 0+$,
\[          \phi_{K_t}(z) =  \Im(z) + {\dot a(0) \, t} \, \Im(1/z) 
      +
                O(tr_t), \]
i.e.,
\[            \p_t[\phi_{t}(z)]|_{t=0+} =   {\dot a(0)} \, \Im(1/z)
  . \]
If the domains $D_t = \Half \setminus K_t$ are simply connected
and $g_t = g_{K_t}$,
\[         \p_t[g_t(z)]|_{t=0} = \frac{\dot a(0)}{z} . \]
\end{proposition}  

The last proposition is the basis for the following proposition
which introduces the {\em (chordal) Loewner differential equation}.
We will not give the details of the proof.  One does need to
prove that $U_t$ is continuous; the Beurling estimate is a useful
tool for this.

\begin{proposition}
Suppose $\gamma:(0,T] \rightarrow \Half$ is a simple
curve with $\gamma(0+) := U_0 \in \R$.  Let $a(t) =
\hcap(\gamma(0,t]),g_t = g_{\gamma(0,t]}$, and suppose
that $a$ is $C^1$.  Then
$g_t(z)$ satisfies
\begin{equation}  \label{chordalequationgen}
            \dot g_t(z) = \frac{\dot a(t)}{
                   g_t(z) - U_t} , \;\;\;\;
   g_0(z) = z , 
\end{equation}
where $U_t = g_t(\gamma(t))$.  For $z \in \Half \setminus
\gamma(0,T]$, this is valid for $t \leq T$.  For $z =
\gamma(s)$, this is valid for $t < s$.  The function
$t \mapsto U_t$ is continuous.
 \end{proposition}

\begin{figure}[htb]
\begin{center}
\epsfig{file=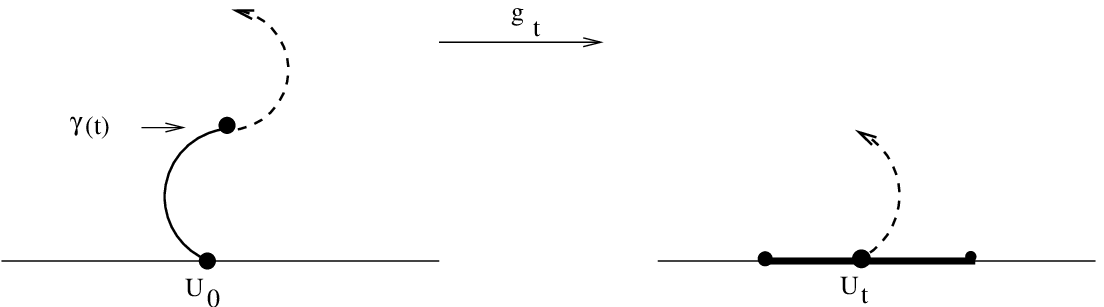}
\caption{The conformal transformation $g_t$ induced by $\gamma$ .}
\end{center}
\end{figure}

In the last proposition we started with a curve and 
produced a function $U_t$.  We will  reverse
the procedure here.
Suppose $a:[0,\infty) \rightarrow
[0,\infty)$ is a strictly increasing
$C^1$ function with $a(0) = 0$, and $U:[0,\infty)
\rightarrow \R$ is a continuous function.
We will consider the (chordal) Loewner equation.
\beq  \label{chordalloewner}
  \dot g_t(z) = \frac{\dot{a}(t)}{g_t(z) - U_t}, \;\;\;\;
    g_0(z) = z, 
\end{equation}
For each $z \in \C \setminus \{0\}$, the solution of the equation above
exists up to time $T_z \in (0,\infty]$. Using the
continuity of $U_t$, one can see that for every $\epsilon > 0$
there is a $t > 0$ such that $T_z \geq t$ for $|z-
U_0| \geq \epsilon$. Moreover, it can be shown that
 for fixed $t$,
$g_t$ is the conformal transformation of $H_t :=\{z \in \Half:
T_z > t\}$ onto $\Half$ with expansion at infinity
\[    g_t(z) = z + \frac{a(t)}{z} + O(|z|^{-2}), \;\;\;\;
  z \rightarrow \infty . \]
If $f_t(z) = g_t^{-1}(z)$, then $f_t$ is a conformal
transformation of $\Half$ onto $H_t$. By differentiating
both sides of 
  $f_t(g_t(z)) = z$ with respect to $t$, we see that
$f_t$ satisfies  
\[        \dot f_t(z) = - \frac{\dot a(t) \, f_t'(z)}
                             {z- U_t} , \;\;\;\;
   f_t(0) = 0 . \]
Here, and throughout these lectures, $'$ refers to spatial
derivatives. 

We have made no assumptions on $U_t$ other than continuity.
Since $U_t$ is real-valued, $g_t(\overline z) = \overline
{g_t(z)}$ and $T_{\overline z} = T_z$. Usually we consider
the equation only for $z \in \overline \Half $.
If we  write
\[  g_t(z) = u_t(z) + i v_t(z), \]
then (\ref{chordalloewner})  becomes
\[     \dot{u}_t(z) = \frac{\dot a(t) \, (u_t(z) - U_t)}{(u_t(z)
   - U_t)^2 + v_t(z)^2}, \;\;\;\; \dot{v}_t(z) =
  -  \frac{\dot a(t) \, v_t(z)}{(u_t(z)
   - U_t)^2 + v_t(z)^2}. \]
For fixed $z$, we will often write 
\[    Z_t = Z_t(z) = g_t(z) - U_t  = X_t + i Y_t, \]
\[    X_t = X_t(z) = u_t(z) - U_t, \;\;\;\;
        Y_t = Y_t(z) = v_t(z) , \]
in which case we can write \eqref{chordalloewner}   as
\begin{equation}  \label{dec22.1}
  \dot{u}_t(z) = \frac{\dot a(t) \, X_t}{X_t^2 + Y_t^2},
\;\;\;\;\; \dot v_t(z) = \dot Y_t = -\frac {\dot a(t)\, Y_t}
   {X_t^2 + Y_t^2}  . 
\end{equation}
Differentiating \eqref{chordalloewner} with respect to
$z$ gives
\[             \dot g_t'(z) = - \frac{\dot a(t) \, g_t'(z)}
 {
                                Z_t^2}. \]
Since $g_0'(z) = 1$, we can solve this equation,
\begin{equation}  \label{chordalderivative}
     g_t'(z) = \exp\left\{-\int_0^t \frac{\dot a(s) \, ds}
                {Z_s^2}\right\}, 
\end{equation}
\begin{equation}  \label{chordalabsolutederivative}
   |g_t'(z)| =  \exp\left\{
-\int_0^t  \Re\left[\frac{\dot a(s) }
                {Z_s^2}\right]\, ds\right\}
   = \exp\left\{
 \int_0^t   \frac{\dot a(s) \, (Y_s^2 - X_s^2)\, ds}
                {(X_s^2 + Y_s^2)^2} \right\}
\end{equation} 
The last equation can be rewritten as
\begin{equation}  \label{may4.1}
      \p_t |g_t'(z)| = \dot a(t) \, |g_t'(z)| \,
                           \frac{Y_t^2 - X_t^2}{(X_t^2 + Y_t^2)^2}. 
\end{equation}

\section{Maps generated by a curve}

Any continuous function $U_t$ and $C^1$ function
$a(t)$ produces 
the conformal maps $g_t$ and hence the domains $H_t$.  It is
not always true that the domains $H_t$ are obtained by slitting
$\Half$ with a curve $\gamma$. By a
{\em curve}, we will mean  a continuous
function from an interval in $\R$ into $\C$.

\begin{definition}
Let 
\[  H_t^{\rm pion} := \bigcup_{0 \leq s \leq t}
 \p H_s\]
 denote the {\em pioneer points} of
$H_t$. 
  If there is a curve
$\gamma:[0,\infty) \rightarrow  \overline \Half$ with $\gamma(0)
\in \R$ such that
\[     H_t^{\rm pion} = \R \cup \gamma(0,t], \]
we say that $g_t$ is {\em generated by the
curve $\gamma$}. (The term pioneer
comes from the idea that a pioneer
is someone who is on the frontier at some time.)
\end{definition}

Note that $H_t$ is the unbounded component of 
$\Half \setminus H_t^{\rm pion}$.   Suppose
that $g_t$ is generated by a curve $\gamma$.
If  $\gamma$ is simple with $\gamma(0,t] \subset
\Half$,
then $H_t^{\rm pion} = \R \cup \gamma(0,t]$,
$H_t = \Half \setminus \gamma(0,t]$.  If $\gamma$
is not simple, then it is possible for there 
to be points in $\Half \setminus H_t$ that
are not on $\gamma(0,t]$. Since $\p H_t
\subset \R \cup \gamma(0,t]$, we can see that
\[  \hcap(\gamma(0,t]) = \hcap (\Half \setminus
  H_t) =  a(t) . \]
In particular, for every $t, \epsilon > 0$,
\[  \gamma (t, t+\epsilon] \cap H_t \neq
\emptyset . \]
Also,
\[   U_t = g_t(\gamma(t)) = \lim_{\epsilon
\rightarrow 0+} g_t(\gamma(t + \epsilon)).\]
It is sometimes difficult to tell whether or not the
maps are generated by a curve.  The next proposition gives
a criterion.

\begin{proposition}  Suppose that 
 $U_t: [0,1] \rightarrow \R$ is a continuous
function and $g_t$ is the solution to \eqref{chordalloewner}
with $a(t) =  t$.
Suppose that there exists $v:(0,1] \rightarrow (0,1]$
with $v(0+) = 0$ such that for all $0 \leq t \leq 1$
and all $\epsilon < 1$,
\begin{equation}  \label{jan2.1}
       \int_0^\epsilon |f_t'(U_t + iy)| \, dy
               \leq v(\epsilon). 
\end{equation}
Then $U_t$ is generated by a curve $\gamma$.
\end{proposition}

\begin{proof}  Using \eqref{jan2.1}, we can see that
the limit
\begin{equation}  \label{jan2.1.1}
          \gamma(t) := \lim_{\epsilon \rightarrow 0+}
             f_t(U_t + i\epsilon) , 
\end{equation}
exists and
\begin{equation}  \label{jan2.2}
             |\gamma(t) - f_t(U_t + i \epsilon)| \leq
              v(\epsilon), \;\;\;\;\; 0 < t,\epsilon \leq 1. 
\end{equation}
For fixed $\epsilon > 0$, the function $t \mapsto
f_t(U_t + i \epsilon)$ is continuous and hence there
is a $\delta_\epsilon$ such that
\[    |f_t(U_t+ i \epsilon) - f_s(U_s
  + i \epsilon)| \leq v(\epsilon), \;\;\;\;
  |s-t| \leq \delta_\epsilon , \]
which implies
\[      |\gamma(t) - \gamma(s)| \leq 3v(\epsilon),
\;\;\;\;  |s-t| \leq \delta_\epsilon. \]
This shows that $t \mapsto \gamma(t)$ is a continuous
function.
The definition of $\gamma$ shows that $\gamma(t) \in \p H_t$ and
hence $H_t$ is contained in the unbounded component of
$\Half \setminus \gamma(0,t]$. It is not difficult to show,
in fact, that these are equal. 

\end{proof}

\begin{remark}  The uniform bound \eqref{jan2.1} is more than
is needed show that the limit \eqref{jan2.1.1} exists, 
but it is used to prove
  the continuity of $\gamma$.  There exist examples where
the limit \eqref{jan2.1.1} exists for all $t$ but for which 
$\gamma$ is not a continuous function of $t$.
\end{remark}

\begin{remark}  The distortion theorem tells us that
\[  \int_0^{2^{-n}} |f_t'(U_t + i y)|
   \asymp   \sum_{j=n}^\infty \, 2^{-j} \, |f_t'
                     (U_t + i 2^{-j})|. \]
A sufficient condition is to show that
$|f_t'(U_t + iy)| \leq \phi(y) $ where $\phi$ satisfies
\[             \sum_{j=1}^\infty 2^{-j} \, \phi(2^{-j})
  < \infty. \]
\end{remark}

\section{A flow on conformal maps}

The Loewner equation can be considered as a flow
on the space of locally real conformal transformations at the origin.
Suppose 
\[            F(z) = \sum_{n=0}^\infty \frac{q_n}{n!} \, z^n , \]
is a  function analytic in a neighborhood ${\mathcal N}
= {\mathcal N}_F$
of the 
origin with $q_1 > 0$ and $q_n \in \R$ for $n \geq 0$. Assume
for ease that $U_0 = 0$.  Let $K_t = \Half \setminus
H_t$.   For $t$
sufficiently small, $K_t \subset {\mathcal N}$
and hence we can define $K^*_t = F(K_t)$.  Let $g^*_t
= g_{K_t^*}$ and set
$    \psi_t(z) = (g^*_t \circ F \circ g_t^{-1})(z).$
  We can write
down the differential equation for $\psi_t$.  Assume that
$a(t) = t, \dot a(t) = 1$.  Then the map $g^*_t$
satisfies the equation
\[     \dot g_t^*(z) = \frac{\psi_t'(U_t)^2}{g_t^*(z) - U_t^*}, \]
where $U_t^* = \psi_t(U_t)$.  The extra term $\psi_t'(U_t)^2$ arises
from the scaling rule $\hcap(rK) = r^2 \hcap (K)$. Since
$\psi_t(z) = g_t^* \circ F \circ f_t$, the chain rule gives
\[    \dot \psi_t(z) = \frac{\psi_t'(U_t)^2} {\psi_t(z) - \psi_t(U_t^*)}
     - \frac{\psi_t'(z)}{z - U_t} . \]
In particular, since $U_0 = 0$,
\[    \dot \psi_0(z)=  (\Lambda F)(z) :=
 \frac{F'(0)^2}{F(z)-F(0)} - \frac{F'(z)}{z}. \]
Note that $\Lambda F$ is analytic in ${\mathcal N}$ with
$\Lambda F(0) = -3F''(0)/2$.  By differentiating this
equation, we can find $\dot \psi_0^{(k)}(n)$ for all positive
integers $k$.

\begin{lemma} 
If 
\[   F(z) = \sum_{n=0}^\infty  q_n \, x^n, \]
then
\begin{equation}  \label{oct17.5} \Lambda  F = 
-\frac{3q_2}{2} +
   \left(\frac{q_2^2}{4q_1}-  \frac{2q_3}{3} \right) \, z
     + \left(  \frac{q_2q_3}{6q_1} - \frac{5q_4}{24} - 
\frac{q_2^3}{8q_1^2}  \right) \, z^2 + \cdots. 
\end{equation}
\end{lemma}

\begin{proof}
Since $\Lambda(rF + q_0) = r \Lambda F$, it suffices
to prove the expansion for 
$q_0= 0, q_1 = 1$, for which
\[ \Lambda F(z) = \left[\sum_{n=1}^\infty \frac{q_n}{n!}
  \right]^{-1} - \frac 1z  \sum_{n=0}^\infty \frac{q_{n+1}}{n!}
    z^n   . \]
We expand
\[ \left[\sum_{n=1}^\infty \frac{q_n}{n!}z^n
  \right]^{-1} = \frac 1z \, \left[1 - \left(\sum_{n \geq 2}
     \frac{q_n}{n!}z^{n-1}\right) + \left(\sum_{n \geq 2}
     \frac{q_n}{n!}z^{n-1}\right)^2 - \cdots \right],\]    
\[  \frac 12  \left(\sum_{n \geq 2}
     \frac{q_n}{n!}z^{n-1}\right)  = \frac{q_2}{2} + 
  \frac{q_3}{6} z + \frac{q_4}{24} z^2 + \cdots, \]    
\[ \frac 1z \left(\sum_{n \geq 2}
     \frac{q_n}{n!}z^{n-1}\right)^2  =
     \frac{q_2^2}{4} z + \frac{q_2q_3}{6} z^2 +
     \cdots, \]
\[  \frac 1z \left( \sum_{n \geq 2}
     \frac{q_n}{n!}z^{n-1}\right)^3 = \frac{q_2^3}{8} z^2
     + \cdots , \]
which gives
\[ \frac{1}{F(z)} = \frac 1z - \frac{q_2}{2}
  + \left(  \frac{q_2^2}{4}-  \frac{q_3}{6} \right) \, z
   + \left(  \frac{q_2q_3}{6} - \frac{q_4}{24} - 
\frac{q_2^3}{8}  \right) \, z^2 + \cdots \;\;\;\;.\]
Also,
\[ \frac 1z  \sum_{n=0}^\infty \frac{q_{n+1}}{n!}
    z^n   = \frac{1}{z}
  + q_2  + \frac{q_3}{2} \, z + \frac{q_4}{6} \, z^2 + \cdots,
\]
giving
\[  \Lambda F(z) = -\frac{3q_2}{2} +
   \left(\frac{q_2^2}{4}-  \frac{2q_3}{3} \right) \, z
     + \left(  \frac{q_2q_3}{6} - \frac{5q_4}{24} - 
\frac{q_2^3}{8}  \right) \, z^2 + \cdots \;\;\;\;.\]
\end{proof}

\section{Doubly infinite time}

If $U_t:(-\infty,\infty) \rightarrow \C$ is a continuous
function, we can consider the solution $g_t$ of
the Loewner equation
\begin{equation} \label{reversechordal}
          \dot g_t(z) = \frac{a}{g_t(z) - U_t},
\;\;\;\;  g_0(z) = z . 
\end{equation}
Let $\tilde g_t(z) = g_{-t}(z), \tilde U_t = U_{-t}$.
Then $\tilde g_t, 0 \leq t < \infty$, satisfies
\[     \dot{\tilde g}_t(z) = - \frac{a}{\tilde g_t(z) -
  \tilde U_t} , \;\;\;\;  \tilde g_0(z) = z. \] 

\begin{proposition} \label{doubleprop}
 For each $t \geq 0$, $\tilde g_t$
is a conformal transformation of $\Half$ onto a
subdomain $\tilde H_t =
\tilde g_t(\Half)$ with $\hcap(\Half \setminus H_t) =
 at$ satisfying
\[            \tilde g_t(z) = z + \frac{at}{z} + O(|z|^{-2}),
\;\;\;\;  z \rightarrow \infty . \]
\end{proposition}

\begin{proof}
 For
fixed $T>0$, let $V_t = \tilde U_{T-t}-U_T$.  If $z \in
\Half$, then $r_t(z) =
\tilde g_{T-t}(z)-U_T$ satisfies
\[      \dot r_t(z) = \frac{a}{r_t(z) - V_t},\;\;\;\;\;
   r_0(z) = \tilde g_{T}(z) -U_T, \;\;\;  r_T(z) = z -U_T. \]
In other words, if we let $g_t(z) = \tilde g_{T-t}(\tilde g_T^{-1}(z  
 +U_T))
  -U_T$,
then $g_t(z)$ satisfies
\[           \dot g_t(z) = \frac{a}{g_t(z) - V_t}, \;\;\;\;
     g_0(z) = z . \]
This is the usual Loewner equation.  Note that
\[ g_T(z) =\tilde g_0(\tilde g_T^{-1}(z+U_T)) - U_T=
  \tilde g_T^{-1}(z-V_T) + V_T .\]
In particular,
\begin{equation}  \label{oct26.1}
  \tilde g_T'(z-V_T) = (g_T^{-1})'(z) . 
\end{equation}

\end{proof}

\section{Distance to boundary}

Suppose $g_t$ satisfies \eqref{chordalloewner},
$H_t = \{z: T_z > t\}$, and $z \in \Half$.
Recall that 
$ H_t^{\rm pion} := \bigcup_{0 \leq s \leq t}
 \p H_s$
 denotes the  pioneer points  of
$H_t$.  If $\gamma$ is generated by a curve $\gamma$,
then $ H_t^{\rm pion} = \R \cup \gamma(0,t]$.  
In this section we consider
\[   \dist\left[z, H_t^{\rm pion}\right]. \]
If the maps are
 generated by the curve $\gamma$, this is the
same as $\dist[z,\R \cup \gamma(0,t]).$
If $t < T_z$, this is also the same as $\dist(z,
\p H_t)$.  Therefore,
\[    \dist\left[z, H_\infty^{\rm pion}\right]
                = \lim_{t \rightarrow T_z-}
                \dist[z, \p H_t]. \]
This exact quantity is not as easy to study as
a 
 closely
related quantity.  For $t < T_z$, we define
\begin{equation}  \label{distsub}
    Y_t = Y_{t,z} = \Im[g_t(z)], \;\;\;\;
   \distsub_t = \distsub_{t,z} = \frac{Y_t}{|g_t'(z)|}.        
\end{equation}
Using  \eqref{dec22.1} and \eqref{chordalabsolutederivative} we
see that
\begin{equation}  \label{dec26.1}
             \dot \distsub_t = - \distsub_t \,
   \frac{2\, \dot a(t) \, Y_t^2}{(X_t^2 + Y_t^2)^2},  
\end{equation}
\[              \distsub_t =\Im[z]
  \, \exp\left\{ -2 \int_0^t \frac{\dot a(s) \, Y_s^2\, ds}
               {(X_s^2 + Y_s^2)^2} \right\}. \]
In particular, $\distsub_t$ is decreasing in $t$, so we can define
\[        \distsub_{\infty,z}  = \distsub_{T_z-,z} =
   \exp\left\{ -2 \int_0^{T_z} \frac{\dot a(s) \, Y_s^2\, ds}
               {(X_s^2 + Y_s^2)^2} \right\}, \;\;\;\;
  t \geq T_z. \]

The following lemma is an immediate corollary of the
Koebe-$(1/4)$ theorem.

\begin{lemma}  Under the assumptions above, if $t < T_z$,
\begin{equation}  \label{distsub2}
      \frac{\distsub_{t,z}}{4}
                \leq  \dist(z,   H_t^{\rm pion})
    \leq 4 \, \distsub_{t,z}. 
\end{equation}
Hence,
\[         \frac{\distsub_{\infty,z}}{4}
                \leq  \dist(z,   H_\infty^{\rm pion})
    \leq 4 \, \distsub_{\infty,z}. \]     
\end{lemma}

The quantity $\distsub_t$ is sometimes called the {\em conformal radius}.
Note that \eqref{dec26.1}
can be rewritten as
\[          \p_t \distsub_t = - 2\dot a(t) \, \distsub_t
            \, [\pi \,H_\Half(0,Z_t)]^2 , \]
where $H$ denotes the Poisson kernel.

\section{Exercises for Lecture 2}

\begin{exercise}  \label{may3.exer0}
 Let
\[   D=\{x+iy: 0 < x < \infty, 0 < y < \pi \} 
\] be a half-infinite rectangle.  Use separation of
variables, as outlined below, to find $H_D(x+iy,iy')$
for $x > 0 , 0 < y,y' < \pi$.
\begin{enumerate}
\item  Find all
 functions of the form
\[               \phi (x+iy) = \phi_{ 1}(x) \, \phi_{ 2}(y) \]
that are harmonic in $D$ and vanish on the horizontal
boundaries of $D$. 
\item  Find the linear combination of these functions whose
boundary value is the $\delta$-function at $iy'$.
\end{enumerate}
\end{exercise}

\begin{exercise}  \label{may3.exer1}  Use Exercise
\ref{may3.exer0} and conformal invariance
to justify \eqref{may3.3} and \eqref{may3.4}.
\end{exercise}

\begin{exercise} \label{may3.exer2}
 Let $B_t$ be a standard complex Brownian motion
starting at the origin and $\tau = \tau_\Disk = \inf\{t:
|B_t| = 1\}$.  Find
\[         \lim_{\epsilon \rightarrow 0}
  \, \epsilon^{-1/2} \, \Prob\{B[0,\tau] \cap [\epsilon,1] =
  \emptyset \}. \]
\end{exercise}

\begin{exercise}  \label{exer.excursion}
Let $  D ,\phi_D$ be as in the beginning of
Section \ref{hcapsec}.  Let $B_t$ be a standard complex Brownian
motion starting at $z \in D$.  For each $R > 0$, let
\[         \sigma_R = \inf\{t: B_t \not\in D \mbox{ or }
              \Im[B_t] = R \}. \]
Show that
\[             \lim_{R \rightarrow \infty}
             R \, \Prob\{\Im[B_{\sigma_R}] = R \} =
           \phi_D(z) . \]
Conclude that if $x \in \R$ and $\dist(x, \Half \setminus D) > 0$,
then $\p_y \phi_D(x)$ is the probability that a Brownian motion
started at $x$ conditioned to stay in $\Half$ forever (i.e., 
a Brownian excursion) stays in $D$ for all time.
\end{exercise}

\begin{exercise}
Find $g_K$ for the following sets:
\begin{itemize}
\item  $K = (0,yi]$
\item  $K$ is the line segment from $0$ to $e^{i\theta}$.
\end{itemize}
\end{exercise}

\begin{exercise}  Suppose $g_t$ is the solution to
\eqref{chordalequationgen} with $a(t) = at$.
  Fix $T > 0$ and
let $V_t = U_{T-t} - U_T$.  Suppose $h_t$ is the
solution to the reverse time Loewner equation
\[               \dot h_t(z) = \frac{a}{V_t - h_t(z)}
       , \;\;\;\; 0 \leq t \leq T. \]
Show that $h_T(z) = g_T^{-1}(z+U_T)$.
\end{exercise}

\lecture{Schramm-Loewner evolution (SLE)}

We now return to the problem of determining possible
candidates for the scaling limit of discrete systems.  
We will focus on $\mu_D^\#(z,w)$, and we will not
worry about the parametrization. We start by considering
$\mu_\Half^\#(0,\infty)$.  If we parametrize the curve
so that the half-plane capacity grows linearly, then
we get conformal maps satisfying
 \begin{equation}  \label{chordalspec}
   \p_t g_t(z) = \frac{a}{g_t(z) - U_t} , \;\;\;\;
       g_0(z) = z , 
\end{equation} 
where $U_t$ is now random.  Conformal invariance and the domain
Markov property translate into conditions on $U_t$.  In fact,
they require $U_t$ to be continuous with stationary, independent
increments. It is well known that this implies that $U_t$ is
a one-dimensional Brownian motion. 

Since the process should be invariant under dilations of $\Half$,
we can see that $h_t(z) :=
r^{-1} \, g_{r^2t}(rz)$ should have the same
distribution as $g_t(z)$.   Note that if $g_t$ satisfies
\eqref{chordalspec}, then
\[   \p_t \, h_t(z) = \frac{a}{h_t(z) - U_t^*} , \]
where $U_t^* = r^{-1} \, U_{r^2t}$.  If $U_t$ is a Brownian
motion, then $U_t^*$ has the same distribution as $U_t$ provided
that the drift is zero.  If the drift is nonzero, they do not
have the same distribution. 
  
  We can choose the variance
of $U_t$ and we can choose the parameter $a$.  A simple time
change shows that, in fact, there is only one free parameter.
As originally defined, the parameter $a$ was chosen to be $2$
and $\kappa$ was used for the variance of the Brownian motion.
Here, we choose the variance of the Brownian motion to be $1$
and use $a$ as the free parameter.  Choosing $a=2/\kappa$ gives
$SLE_\kappa$.

\section {Definition}

\begin{definition}
The {\em chordal Schramm-Loewner evolution (from
$0$ to $\infty$ in $\Half$ parametrized so that
$\hcap(\gamma(0,t] ) = at$)} with parameter
$\kappa = 2/a$ is the solution of \eqref{chordalspec}
where $U_t = - B_t$ is a standard one-dimensional
Brownian motion with $B_0 = 0$.  The (random) curve
$\gamma$ that generates the maps $\{g_t\}$ is called
the $SLE_\kappa$ curve.
\end{definition}

It is not immediately obvious but has been proved that
$SLE_\kappa$ is generated by a curve.

If $z \in \Half$ and we write $Z_t(z) = X_t + i Y_t =
g_t(z) - U_t$, then \eqref{dec22.1} gives
\begin{equation}  \label{dec26.2}
dX_t = \frac{aX_t}{X_t^2 + Y_t^2} \, dt + dB_t, \;\;\;\;
         \p_tY_t = -\frac{aY_t}{X_t^2 + Y_t^2} =
   Y_t \, \frac{-a X_t^2 - aY_t^2}{(X_t^2 + Y_t^2)^2} . 
\end{equation}
We also let
\[  \distsub_t = \distsub_t(z) =
\frac{Y_t}{|g_t'(z)|} ,\;\;\;\;\;
   R_t = R_t (z) = \frac{X_t}{Y_t} , \]
\[  \Theta_t = \Theta_t(z)
  = \arg(Z_t) , \;\;\;\;
 \O_t = \O_t(z) = (R_t^2 + 1) = [\sin^2 \Theta_t]^{-1} , \]
Recall that $\distsub_t$ is related to the distance between
$z$ and $\R \cup \gamma(0,t]$.  
Using \eqref{may4.1}, we see that
\[              \p_t \distsub_t =  -\distsub_t \, \frac{2a \,Y_t^2}
                  {(X_t^2 + Y_t^2)^2}   .\]
It\^o's formula gives
\begin{equation}  \label{Theta}
 d\Theta_t = \frac{(  1- 2 a) \, X_t \, Y_t}{(X_t^2 + Y_t^2)^2}
  \, dt - \frac{ Y_t}{X_t^2 + Y_t^2} \, dB_t, 
\end{equation}
\[   d\O_t^{r} =\O_t^r
  \, \left[\frac{[2r^2 + (4a-1)r] \, X_t^2 + r\, Y_t^2}
    {(X_t^2 + Y_t^2)^2} \, dt + \frac{2r
  \, X_t}{X_t^2 + Y_t^2} \, dB_t \right].\]
It is worth remembering that $Y_t,\distsub_t$ are
differential functions of $t$ and the formulas for them are valid
for any driving function $U_t$.  However, $X_t,R_t,\O_t$ have
non-trivial quadratic variation.
If we let
\[   N_t = \distsub_t^{-u/a} \, Y_t^{\theta/a} \, \O_t^r,\]
then the product rule gives
\[  dN_t = N_t \, \left[\frac{[2r^2 + (4a-1)r-\theta]\,X_t^2 + 
   [2u+r-\theta]\, Y_t^2}
   {(X_t^2 + Y_t^2)^2} \, dt  + \frac{2r \, X_t}{X_t^2 + Y_t^2}
\, dB_t\right]. \]
If the $dt$ term is zero, this is a local martingale.  Hence
we get the following.

\begin{proposition}  If $r \in \R$ and 
\[          u = u(r) = r^2 + (2a-1) \, r , \]
Then
\[    M_t = M_t(z) = \distsub_t^{-\frac ua} \, Y_t^{\frac{2u+r}{a}}
  \, O_t^r , \]
is a local martingale satisfying
\[                   dM_t =  \frac{2rX_t}{X_t^2 + Y_t^2}
  \, M_t \, dB_t. \]
\end{proposition} 
 
\begin{example}  \label{may4.example1}

  $r =  \frac 12 - 2a, u = a - \frac 14$ gives the local martingale
\[            M_t = \distsub_t^{d-2} \, \O_t^{\frac 12 - 2a} , \]
where
\[               d = 1 + \frac 1{4a} = 1 + \frac \kappa 8. \]
\end{example}

\begin{example} \label{may4.example2}

   $r= -2a, u = 2a$ gives the local martingale
\[           M_t = |g_t'(z)|^2 \, O_t^{-2a}. \]
\end{example}

\begin{example} \label{may4.example3}

   Let 
\[  
     r = -b = \frac{1-3a}{2}, \;\;\; u = a\tilde b=\frac{b\,(1-a)}{2 } . \]
    This gives the
martingale
\begin{equation}  \label{may15.2}
          M_t =  
\distsub_t^{-\tilde b} \, Y_t^{-b} \, O_t^{-b} =
|g_t'(z)|^{\tilde b} \, Y_t^{-\tilde b} \, [\pi H_\Half(0, Z_t)]^b . 
\end{equation}
Here $H_\Half$ denotes the Poisson kernel in $\Half$,
\[              H_\Half(0,x+iy) = \frac{y}{\pi(x^2 + y^2)} . \] 

\end{example}

\section{Phases}

Recall that a curve is simple if it has no self-intersections.

\begin{theorem} \label{simpletheorem}
 If $\kappa \leq 4$, $SLE_\kappa$ paths 
are simple.  If $\kappa > 4$, $SLE_\kappa$ paths have 
self-intersections.  In fact, if $\kappa > 4$ for every $s < t$
there exist $s < t_1 < t_2 < t$ with $\gamma(t_1) = \gamma(t_2)$.
\end{theorem}

\begin{remark} To be more precise, the  theorem states 
 that the facts hold ``with probability one''.  We will
feel free to leave this phrase out of statements of theorems.
\end{remark}

We will prove a related fact and leave the proof of the theorem
as an exercise (Exercise \ref{ex.phase}).

\begin{proposition}  \label{realprop}
If $\kappa \leq 4$, $\gamma(0,\infty) \subset
\Half$.  If $\kappa > 4,$ $\gamma(0,\infty) \cap \R \neq \emptyset$.
\end{proposition}

\begin{proof}  Let $x > 0$.  Then $\gamma(0,\infty) \cap [x,\infty)
\neq \emptyset$ if and only if $T_x < \infty$.
 We will show that $\Prob\{T_x < \infty\}$
is $1$ if $\kappa > 4$ and equals $0$ if $\kappa \leq 4$.  Let $Z_t =
Z_t(x) = g_t(x) - U_t$.  Then the Loewner equations tells us that
\begin{equation}  \label{Bessel}
              dZ_t = \frac{a}{Z_t} \, dt + d B_t,\;\;\;\;
   Z_0 = x . 
\end{equation}
This is the Bessel equation, and it is well known that  
the probability that $Z_t$ reaches the origin in finite time
equals $1$ or $0$ depending on whether $a < 1/2$ or
$a \geq 1/2$.
\end{proof}

\begin{remark}  If $W_t$ is a $d$-dimensional Brownian motion
starting at $z \in \R^d \setminus\{0\}$, then $Z_t = |W_t|$
satisfies \eqref{Bessel} with $a = \frac{d-1}{2}, x = |z|,$
and $B_t$ a standard Brownian motion.
\end{remark}

\begin{theorem} \label{theoremphase2}
  If $\kappa \geq 8$, then $\gamma[0,\infty)
= \overline{\Half}$, i.e., $\gamma$ is plane-filling.  If $\kappa
< 8$, then for each $z \in \Half$, $\Prob\{z \in \gamma(0,\infty)\}
=0 $.
\end{theorem}

\begin{proof} 
Suppose $\kappa  = 2/a > 4$ and let $z \in \Half$.  By using
the previous theorem and scaling we can see that $\Prob\{T_z < \infty\}
 = 1 $.  Suppose $\gamma(T_z) \neq z$.  Then a straightforward
argument (Exercise \ref{ex.straight}) shows that $\Theta_{T_z-}
 \in \{0,\pi\}$, where 
 $\Theta_t = \Theta_t (z). $   
It will be convenient to make the   time change 
\[  \hat X_t = X_{\sigma(t)} , \;\;\;\;  \hat Y_t = Y_{\sigma(t)},
   \]
where
\[        \p_t \sigma(t) = 
  {X_{\sigma(t)}^2 + Y_{\sigma(t)}^2 }   . \]
\[    \p_t  \hat Y_t = \p_t Y_{\sigma(t)}
                = \p_t\sigma(t) \, Y_{\sigma(t)} \, \frac{-a}
  {X_{\sigma(t)}^2 + Y_{\sigma(t)}^2} = - a \, \hat Y_t , \]
i.e., $\hat Y_t = e^{-at}.$  Under this time change $\log \hat Y_t$
is a deterministic linear function.  Time $T_z$ in the usual
parametrization becomes time $\infty$ in the time change.
Under this time change \eqref{Theta} becomes
\[    d \hat \Theta_t = ( \frac 12-a) \, 
  \sin (2\hat \Theta_t) \, dt   +  \sin \hat \Theta_t \, d \hat B_t . \]
for a standard Brownian motion $\hat B_t$.  We have
reduced the problem to a question about a one-dimensional diffusion.
One can show that if $a \leq 1/4$ it is {\em not} the case that
$\sin \hat \Theta_t \rightarrow 0$ as $t \rightarrow \infty$.  We will
not prove this but let us sketch why this is true.  If we do a time
change, this equation looks like
\[     d \tilde \Theta_t = ( \frac 12-a) \, 
  \frac{\sin (2\tilde \Theta_t)}{\sin^2  \tilde \Theta_t} \, dt 
  +  d \tilde B_t . \]
For $\tilde \Theta_t$ near zero this looks like
\[    d \tilde \Theta_t  = \frac{(1 - 2a)}{\tilde \Theta_t} \, dt
  +  d \tilde B_t . \]
by comparison with a Bessel equation, we see that to keep this
from reaching the origin we need $1 - 2a \geq 1/2$ or $a \leq
1/4$.

This is not quite enough to prove the statement for $4 < \kappa < 8$.
Since this follows from Theorem \ref{dimensiontheorem} in the next
section, we will not bother to give the complete details here.
\end{proof}

If $\kappa < 8$, then $\Theta_{T_z-} = \hat \Theta_\infty 
\in \{0,\pi\}$.  Let $\phi(z) = \Prob\{\Theta_\infty = \pi\}$.
Scaling shows that $\phi$ depends only on the angle.  Since
$\phi(\hat \Theta_t)$ is a martingale, we can use It\^o's
formula to conclude that $\phi(\theta)$ satisfies
\[ 2\,(1-2a) \, \phi'(\theta)\, \cos \theta \,   
 +  \phi''(\theta) \, \sin \theta 
  = 0 , \;\;\;\; \phi(0) = 0, \phi(\pi) = 1. \]

\section{Dimension of the path}

In this section we will discuss the dimension of the path $\gamma$.
We will not give all the details.

\begin{theorem}  \label{dimensiontheorem}
There exists $c_*$ such that if $a > 1/4$,
then for all $z \in \Half$
\[           
             \Prob\{\distsub_\infty(z) \leq \delta\}
\sim  c_* \, G(z)
 \, \delta^{2-d},\;\;\;\; \delta \rightarrow 0, \]
where $d = 1 + \frac 1{4a}$
and $G$ is the ``Green's function'' for chordal $SLE_{2/a}$
defined by
\[               G(y(x+i)) = y^{d-2} \, (x^2 + 1)^{\frac 12 - 2a}. \]

\end{theorem}

\begin{proof}
By scaling it suffices to prove this result when $z=x+i$.  We fix
$x$ and we consider the martingale from Example \ref{may4.example1},
\[                 M_t = \distsub_t^{d-2} \, \sin^{4a-1}
  \Theta_t  
                  = |g_t'(z)|^{2-d} \, G(Z_t), \]
which satisfies
\[                   dM_t = \frac{(1-4a) \, X_t}{X_t^2 + Y_t^2}
  \, M_t \, dB_t. \]
One can check that with probability one $M_t \rightarrow 0$.
Let $\tau _\delta = \inf\{t: \distsub_t \leq \delta \}$. For
fixed $\delta,$, $M_{t \wedge \tau_\delta}$ is a 
uniformly bounded martingale, and hence the
optional sampling theorem gives
\[   G(z) =   M_0 = \lim_{t \rightarrow \infty}
  \E[M_{\tau_\delta \wedge t}] =  \E[M_{\tau_\delta};
\tau_\delta < \infty ]  = \delta^{d-2}
  \, \E[\sin^{ 4a - 1} \Theta_{\tau_\delta}; \tau_\delta < \infty].  \]
With the aid of Girsanov (details of this argument are left
as Exercise \ref{jun4.ex1}), we can show that there exists
$0 < c_* < \infty$ such that
\[      \E[\sin^{  4a - 1} \Theta_{\tau_\delta}; \tau_\delta < \infty]
   \sim [1/c_*] \, \Prob\{\tau_\delta < \infty\}. \]
\end{proof}

\begin{remark}
In fact it can be shown (but is more difficult) that the Hausdorff
dimension of the path $\gamma(0,\infty)$ is $d = 1 + \frac \kappa
8$ if $\kappa \leq 8$.
\end{remark}

\section{Cardy's formula}

In this section, we derive a formula for $SLE$ that corresponds
to Cardy's formula for percolation as discussed in Lecture 1,
Section \ref{percsec}.  Suppose $\gamma(0,\infty)$ is an $SLE_\kappa$
curve with $\kappa > 4$.  Suppose $x,y > 0$.  Then with probability
one $T_x, T_{-y} < \infty$. We will consider
\[        \phi(x,y) = \Prob\{T_{-y} > T_x\}. \]
Note that scaling implies that $\phi(x,y) = \phi(y/x)$ where
$\phi(y) = \phi(1,y)$.

\begin{proposition} [Cardy's formula] If $0 < a < 1/2$, then
\begin{equation}
\label{may9.1}
    \phi(y) = \frac{\Gamma(2-4a)}{\Gamma(1-2a)^2}
  \int_0^{\frac{y}{y+1}} \frac{du}{u^{2a} \, (1-u)^{2a}}. 
\end{equation}
\end{proposition}

\begin{proof}  Let $X_t = g_t(1)-U_t, J_t = g_t(-y)-U_t$ and
\[               Z_t = \frac{X_t}{X_t - J_t} . \]
Let $\tau = \inf\{t: Z_t \in \{0,1\}\}$ and
\[    \psi(r) = \Prob\{Z_\tau = 0 \mid Z_0 = r\} . \]
Then $\phi(y) = \psi(\frac 1{y+1}).$  Since
\[   dX_t = \frac{a}{X_t} \, dt + dB_t, \;\;\;\;
            \p_t[X_t - J_t] = \frac{a}{X_t} - \frac a{J_t} , \]
we can see that
\[    dZ_t = \frac{a}{(X_t - J_t)^2} \, \left[\frac1{Z_t} +
  \frac 1{Z_t - 1} \right] \, dt + \frac{1}{X_t - J_t} \, dB_t. \]
We can do a random time change $\sigma$ and see that $\tilde Z_t
:= Z_{\sigma(t)}$ satisfies
\[               d \tilde Z_t = \left[\frac{a}{\tilde Z_t} -
  \frac{a}{1 - \tilde Z_t} \right] \, dt + d \tilde B_t, \]
where $\tilde B_t$ is a standard Brownian motion.  However,
$\psi (\tilde Z_t)$ is a martingale. Using It\^o's formula, we
can see that this implies
\[              \psi''(u) + 2a \, \left[\frac 1 u - \frac{1}{
  1-u} \right] \, \psi'(u) = 0 . \]
Solving this equation with the boundary conditions $\psi(0) = 1,
\psi(1) = 0$ gives \eqref{may9.1}.

\end{proof}

To understand why this corresponds to Cardy's crossing formula
for percolation, consider the percolation exploration process as
in Lecture 1, Section \ref{percsec}.  If we again put the
boundary condition of  black on the negative real axis and white
on the positive real axis then the event $\{T_{-y} > T_1\}$ corresponds
to the event that in the (scaled) percolation cluster there
is a connected component of black sites connecting $[-y,0]$
to $[1,\infty)$.  Hence the ``crossing probability'' for $SLE_\kappa$
is given by
\[ 
 P_\Half([-y,0],[1,\infty))  =
    \phi(y) = \frac{\Gamma(2-4a)}{\Gamma(1-2a)^2}
  \int_0^{\frac{y}{y+1}} \frac{du}{u^{2a} \, (1-u)^{2a}}.  \]
In the case
 $\kappa = 6$, which we will see in the next lecture corresponds
to percolation, this crossing probability   becomes
\[ P_\Half([-y,0],[1,\infty))   = \frac{\Gamma(2/3)}{\Gamma(1/3)^2}
  \int_0^{\frac{y}{y+1}} \frac{du}{u^{\frac 23} \, (1-u)^{\frac 23}}. \]
This may not appear like a very nice formula, but if we map
$\Half$ to an equilateral triangle, we get the simple formula
in Figure 9.

\section{Conformal images of $SLE$}

If $D$ is a simply connected domain and $z,w$ are distinct
points in $\p D$, then the measure $\mu_D^\#(z,w)$ is defined
to be the image of $\mu_\Half^\#(0,\infty)$ under a conformal
transformation $f: \Half \rightarrow D$ with $f(0) = z,
f(\infty) = w$.  Here, and for the rest of
these lectures,  we are considering the probability
measures $\mu^\#_\Half$ as being defined on curves modulo
reparametrization.  The map $f$ is not unique; however, any
other such map $f_1$ can be written as $f_1(z) = f(rz)$
for some $r > 0$.  The invariance of $SLE_\kappa$ under scaling
shows that 
$\mu_D^\#(z,w)$ is well defined.

\begin{figure}[htb]
\begin{center}
\epsfig{file=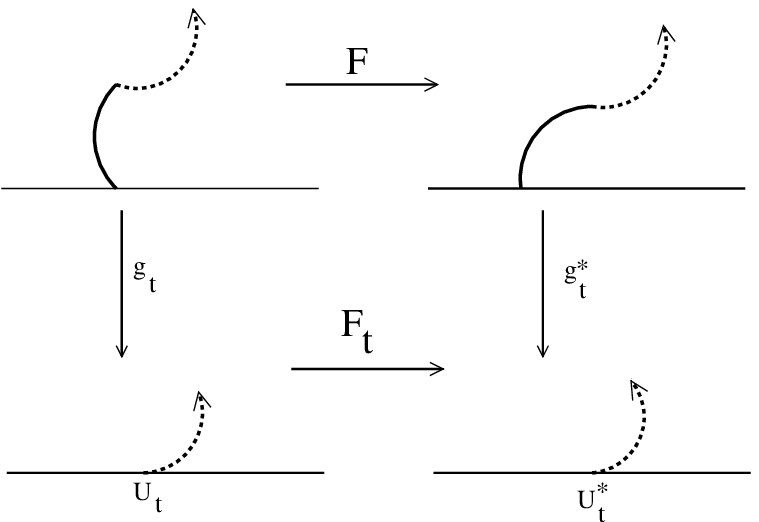}
\caption{The maps $g_t^*,F_t$. Note that
$F_t = g_t^* \circ F \circ g_t^{-1}.$}
\end{center}
\end{figure}

Here we will assume that 
\[  F(z) = \sum_{j=0}^\infty a_j \, z^j , \;\;\;\;
 a_j \in \R, \;\;\; a_1 > 0 , \]
is analytic in a neighborhood of the origin.  The assumptions
on the $a_j$ imply that near $0$ $F$ maps $\R$ into $\R$ and
for some $\epsilon > 0$,
$\epsilon \Disk_+ = \{z \in \Half: |z| < \epsilon\}$ is mapped
conformally into $\Half$. 

Suppose $\gamma$ is an $SLE_\kappa$ curve and let $\tau =
\inf\{t: \gamma(t) \not \in \epsilon \overline{\Disk_+}\}.$
For $t < \tau$, we can define the curve
\[         \gamma^*(t) =  F (\gamma(t)) . \]
We let $H_t^*$ be the unbounded component of $\Half \setminus
\gamma^*(t)$ and $g_t^*$ the unique conformal transformation of
$H_t^*$ onto $\Half$ with
\[             g_t^*(z) = z + \frac{a^*(t)}{z} + \cdots . \]
We also let $F_t = g_t^* \circ F \circ g_t^{-1}$; in other words,
$  F_t \circ g_t = g_t^* \circ F.$  Using only properties of
the Loewner equation, we can see that
\begin{equation}  \label{may26.6}
          \p_t a^*(t) = a \, F_t'(U_t)^2 . 
\end{equation}
The Loewner equation tells us that
\[    \p_t   g_t^*(z)  = \frac{a \, F_t'(U_t)^2}
                       {g_t^*(z) - U_t^*} , \]
where $U_t^* = g_t^*(\gamma^*(t)) =  F_t(U_t) $.

\begin{proposition} \label{prop1}
Under a suitable time change, $\hat U_t
 = U_{\sigma(t)}^*$ satisfies
\[              d \hat U_t = b \, \frac{\Phi_t''(\hat 
 U_t)}{\Phi_t'(
  \hat U_t)} \, dt + dW_t , \]
where $W_t$ is a standard Brownian motion, 
\[              b = \frac{3a-1}{2} = \frac{6 - \kappa}{2\kappa}  \]
is the boundary scaling exponent,  and $\Phi_t =
F_{\sigma(t)}^{-1}$. 
\end{proposition}

\begin{proof}
We use It\^o's formula to get
\[  dU_t^* = dF_t(U_t) = \left[\dot F_t(U_t) + \frac 12
   F''_t(U_t)\right] \, dt + F'(U_t) \, dU_t . \]
The term $\dot F_t(U_t)$ can be calculated using the Loewner 
equation. In fact, it is the first term in \eqref{oct17.5},
i.e., $\dot F_t(U_t) = -(3a/2) \, F''_t(U_t)$ (the factor $a$
appears because \eqref{oct17.5} assumes $a=1$).  Therefore,
\[         d   U_t^* = - b \, {F''_t(U_t)} \, dt
             + F'_t(U_t) \, d U_t. \]
With an appropriate time change we can write this as
\[           d \hat U_t = - b \, \frac{F_{\sigma(t)}''( U_{\sigma(t)})}{F_{\sigma(t)}'( 
    U_{\sigma(t)})^2} \, dt + d W_t. \]
A simple calculation shows that
\[ - \frac{F_{\sigma(t)}''(U_{\sigma(t)}
 )}{F_{\sigma(t)}'(U_{\sigma(t)})^2} = \frac{\Phi_t''(\hat 
 U_t)}{\Phi_t'(
  \hat U_t)}. \]

\end{proof}

\begin{remark} The half-plane capacity is not preserved by $F$.
The time change in the proof is exactly the time change needed so
that $\hcap(\gamma^*(0,\sigma(t)]) = at$.  
\end{remark}

\section{Exercises for Lecture 3}

\begin{exercise}  \label{ex.phase}  Show why
Proposition \ref{realprop} implies Theorem \ref{simpletheorem}
\end{exercise}

\begin{exercise}  \label{ex.straight}  Verify the step
in the proof of Theorem \ref{theoremphase2}.
\end{exercise}

\begin{exercise}  Justify \eqref{may26.6}.
\end{exercise}

\begin{exercise}  \label{jun4.ex1}
The goal is to fill in the details to Theorem 
\ref{dimensiontheorem}.   Let $M_t,\tau_\delta$ be as
in the paragraph following the theorem and assume
$z = x+i, \theta_0 = \arg(z)$.
\begin{itemize}
\item  Show that if we weight by the local martingale
$M_t$ then
\[     d \Theta_t = \frac{2a\, X_t \, Y_t}
   {(X_t^2 + Y_t^2)^2} \, dt - \frac{Y_t}{X_t^2 + Y_t^2} \,
  dW_t , \]
where $W_t$ is a standard Brownian motion in the weighted measure.
\item Show that if we do a random time change $\sigma(t)$
so that $\hat \distsub_t := \distsub_{\sigma(t)} = e^{-2at}$, 
then $\hat \Theta_t = \Theta_{\sigma(t)}$ satisfies
\begin{equation}  \label{jun4.1}
      d \hat \Theta_t = 2a \, \cot \hat \Theta_t
  \, dt + d\hat W_t, 
\end{equation}
where $\hat W_t$ is a standard Brownian motion.
\item  Let
   \[ e(\theta,t) = \E[\sin^{4a-1} \, \hat \Theta_t \mid
  \hat \Theta_0 = \theta], \]
where $\hat \Theta_t$ satisfies  \eqref{jun4.1}.  Use the
Girsanov theorem to prove that
\[      \Prob\{\distsub_\infty \leq e^{-2at} \}
   =  e(t,\theta_0) \, G(z) \, e^{-2at(2-d)}. \]
\item  Find a stationary density for \eqref{jun4.1} and use
it to show that 
\[          \lim_{t \rightarrow \infty} e(t,\theta_0) = c_* \]
for some $c_* \in (0,\infty)$.   Determine $c_*$. 
\end{itemize}
\end{exercise}

\lecture{$SLE_\kappa$ in a simply connected domain $D$}

Let $\domains$ denote the set of simply connected subdomains $D$
of $\Half$ with $\Half \setminus D$ bounded and $\dist(0,\Half
\setminus D) > 0$.
We will write $\Phi_D$ for the unique conformal transformation of
$D$ onto $\Half$ with $\Phi_D(z) = z + o(1)$ as $z \rightarrow
\infty$ (we denoted this by $g_D$ earlier, but it will be convenient
to use a new notation).  We will consider $SLE_\kappa$ from $0$
to $\infty$ in $D$ and show that there are a number of equivalent
ways of defining this process.  We already have one definition: the
image of $SLE_\kappa$ in $\Half$ under the conformal transformation
$F_D := \Phi_D^{-1}$.  We note that (see Exercise \ref{exer.excursion})
 $\Phi_D'(0) \in (0,1]$; in fact $\Phi_D'(0)$
is the probability that a Brownian ``excursion'' in $\Half$ stays
in $D$.

\section{Drift and locality}

Proposition \ref{prop1} can be restated in the following way.
Suppose $\gamma$ is a curve and let $g_t$ be the corresponding
conformal maps.  Let
\[           \tau = \tau_D = \inf\{t: \dist(\gamma(t),
           \Half \setminus D) = 0 \} . \]

\begin{proposition}
Suppose 
 for $t < \tau$, $g_t$ is the solution to the Loewner equation
\eqref{chordalspec} where $U_t$ satisfies the stochastic
differential equation
\[           dU_t = b \, [\log \Phi'_t(U_t)]' \, dt - dB_t. \]
Here $\Phi_t = \Phi_{D_t}$ where $D_t = g_t(D)$. Then $g_t$
(and the corresponding curves $\gamma$) have the distribution
of $SLE_\kappa$ in $D$ stopped at time $\tau$.
\end{proposition}

In other words, the proposition gives an equivalent definition
of $SLE_\kappa$ in $D$, at least up to time $\tau_D$.
The drift is nontrivial unless $b = 0$ which holds if and only
if $a = 1/3, \kappa = 6$.

\begin{theorem} [Locality]  Suppose $\gamma$ is a curve with
the distribution of $SLE_6$ in the domain $D \in \domains$.  
The distribution of $\gamma$ up to time $\tau_D$ is the same
as that of $SLE_\kappa$ in $\Half$ up to time $\tau_D$.
\end{theorem}

If we consider our discrete models, we can see that the only one
for which we would definitely expect the locality property is
the percolation exploration process.  As we saw in the last lecture,
the crossing probability for $SLE_6$ is the same as that
predicted by Cardy for percolation.  

For other values of $\kappa$, we see that $SLE_\kappa$ in $D$
is obtained (at least for small time) by putting a drift in the Brownian
motion.  The Girsanov theorem tells us that Brownian motion with
drift is absolutely continuous with respect to Brownian motion
without drift, at least for small times.

\section{Girsanov}  \label{girsanovsec}

The Girsanov theorem is a very useful tool for studying the
(local) martingales arising in $SLE_\kappa$ so we will discuss
what we need here.  Suppose that $M_t$ is a nonnegative
continuous local martingale satisfying
\begin{equation}  \label{may4.5}
                dM_t = J_t \, M_t \, dB_t,\;\;\;\;\;
  M_0 = 1 .
\end{equation}
Although $M_t$ may not be a martingale, we can approximate $M_t$
by a uniformly bounded martingale.  Indeed, if $\tau_n = \inf\{t:
M_t \geq n\}$ and $M_{t,n} = M_{t \wedge \tau_n}$, then for fixed
$n$,  $M_{t,n}$
is a uniformly bounded martingale satisfying
\[                 dM_{t,n} = J_t \, 1\{\tau_n > t\} \, M_{t,n}
  \, dB_t, \]
and for fixed $t$, $M_{t,n} \rightarrow M_t$ as $n \rightarrow \infty$.
For any nonnegative martingale $N_t$ satisfying
\[               dN_t =  J_t \, N_t \, dB_t, \]
 there is a measure $Q$ defined
by
\[                 Q(V) = \E[1_V \, N_t] , \]
if $V$ if $\F_t$-measurable.  The Girsanov theorem states that under
the measure $Q$,
\[                  W_t = B_t - \int_0^t \, J_s \, ds , \]
is a standard Brownian motion.  In other words, $B_t$ satisfies
\[                     dB_t = J_t \, dt + dW_t , \]
where $W_t$ is a Brownian motion with respect to $Q$.  

The Girsanov theorem requires that $N_t$ be a martingale.  If we only
know that $M_t$ is a local martingale satisfying \eqref{may4.5}, then
we can still use the Girsanov theorem as long as we run the paths
up to a stopping time $\tau_n$.  In this case we get the equation
\[    dM_t = J_t \, M_t \, dB_t, \;\;\;\; t < \tau_n . \]
If we weight by the (local) martingale $M_t$, we can say that 
$B_t$ satisfies
\[              dB_t = J_t \, dt + dW_t , \;\;\;\; t < \tau_n . \]
One can often use this equation to determine whether or not
$M_t$ is actually a martingale.  Essentially what keeps the local
martingale from being a martingale is the fact that some mass ``goes
to infinity in finite time''.  One can check that this happens by
time $t$ if and only if
\[             \lim_{n \rightarrow \infty}
               \E[M_{t,n}; \tau_n < t] > 0 , \]
i.e., if and only if in the {\em weighted} measure, there is a positive
probability of explosion of $M_t$ in finite time.

\begin{remark}  Although we will not prove the Girsanov theorem
here, let us give a heuristic reason why it is true.  Imagine
that $\Delta B_t :=
B_{t + \Delta t} - B_t$ were equal to $\sqrt{\Delta t}$
with probability $1/2$ and $-\sqrt{\Delta t}$ with probability
$1/2$.  Then $M_{t+\Delta t}$ is about $(1 + J_t \sqrt {\Delta t})
M_t$ with probability $1/2$ and $(1 - J_t \sqrt {\Delta t})
M_t$ with probability $1/2$.  If we weight by $M_{t + \Delta t}$,
then in the weighted measure $\Delta B_t$ equals $\sqrt {\Delta t}$
with probability $(1 + J_t \sqrt {\Delta t})/2$ and equals
$-\sqrt {\Delta t}$ with probability $(1 - J_t \sqrt {\Delta t})/2$.
Under this measure
\[   \E[\Delta B_t] = J_t \, \Delta t  . \]
In other words, by weighting by the martingale we obtain a drift
of $J_t$.
  
\end{remark}

\section{The restriction martingale}  \label{restmartsec}

Let $D,\Phi,\Phi_t$ be as above.  We will consider
$\Phi_t'(U_t)^b$.  In the next proposition, we use
the fact that
\[   \dot \Phi_t'(U_t) = a \, \left[\frac{\Phi_t''(U_t)^2}
               {4 \, \Phi_t'(U_t)} - \frac{2\, \Phi_t'''(U_t)}
   3 \right]. \]
This is a property of the Loewner equation and can be seen
as the second term in \eqref{oct17.5}.  With this, the following
proposition is a straightforward It\^o's formula
calculation.   The {\em Schwarzian derivative} of a function
$f$, $Sf$, is defined by
\[              Sf(z) = \frac{f'''(z)}{f'(z)} - \frac{3 \, f''(z)^2}
   {2 \, f'(z)^2}, \]
and the {\em central charge} $\c$ is defined by
\[   \c = \frac{2b(3-4a)}{a} = \frac{(3\kappa - 8) \, (6-\kappa)}
   {2\kappa}. \]

\begin{proposition} \label{feb26.prop1}
If $U_t$ is $SLE_{2/a}$ and $D \in \domains$, 
 then for $t < \tau_D$, 
\[    d[\Phi_t'(U_t)^b]
  = \Phi_t'(U_t)^b \, \left[ \frac{a\c}{12} \, S\Phi_t(U_t) \, dt
         - b\,\frac{\Phi_t''(U_t)}{\Phi_t'(U_t)} \, dB_t
              \right]     . \]
If we let
\[    M_t = \exp\left\{-\c \int_0^t  \frac{aS\Phi_t(U_t)}{12} \, dt
  \right\} \, \Phi_t'(U_t)^b, \]
then $M_t$ is a local martingale satisfying
\[    dM_t = - b\,\frac{\Phi_t''(U_t)}{\Phi_t'(U_t)} \,M_t \,dB_t. \] 
In particular, if $\kappa = 8/3$,
\[        M_t =  \Phi_t'(U_t)^{5/8}, \]
is a martingale.
\end{proposition}

  Then under the weighted
measure $B_t$ satisfies
\[               dB_t =  - b\,\frac{\Phi_t''(U_t)}{\Phi_t'(U_t)} 
  \, dt + dW_t, \]
where $W_t$ is a Brownian motion in the new measure, i.e., $U_t$
satisfies
\[         dU_t = b 
 \,\frac{\Phi_t''(U_t)}{\Phi_t'(U_t)} 
  \, dt - dW_t, \]

\medskip

\noindent {\bf Fact.} {\em  $SLE_\kappa$ weighted by $\Phi_t'(U_t)^b$
gives $SLE_\kappa$ in the smaller domain. }

\medskip

This fact is valid for all $\kappa$ provided that $t$ is sufficiently
small.  For the rest of this lecture we will consider $\kappa \leq 4$
for which the curves are simple.  In Proposition \ref{feb27.prop2}
we will see that 
\[                  S\Phi_t(U_t)\leq 0 , \]
so $M_t \leq 1$ for $\kappa \leq 8/3$.  This implies that $M_t$ is
a martingale.  In fact, for all $\kappa \leq 4$, this is a 
martingale.
 
\medskip

\begin{theorem} [Restriction]  If $\kappa = 8/3$ and $D \in
\domains$, then
 \[  \Prob\{\gamma(0,\infty) \subset D\} = \Phi_D'(0)^{5/8}. \]
Moreover, 
\[          \frac{d\mu_D(0,\infty)}{d\mu_\Half(0,\infty)}
           = 1\{\gamma(0,\infty) \subset D\}. \]
\end{theorem} 

\begin{proof}  We first need to observe (we omit the
argument) that if $\gamma$ is a fixed
curve with $\gamma(t) \rightarrow \infty$ and $\gamma(0,\infty) \subset
D$, then $\Phi_t'(U_t) \rightarrow 1$.  For each $\epsilon > 0$,
let $\rho_\epsilon = \inf\{t: \dist(\gamma(t), \Half\setminus D)
 \leq \epsilon \}$.

 Since $M_{t,\epsilon}  = \Phi_{t \wedge \rho_\epsilon}
  (U_{t \wedge \rho_\epsilon})^{5/8}$
is a uniformly bounded martingale, 
\[   \Phi_D'(0)^{5/8} = \E[M_0] = \E[M_{\infty,\epsilon}]
     = \E[M_\infty; \rho_\epsilon = \infty] +
   \E[M_{\rho_\epsilon}; \rho_\epsilon < \infty] . \]
But we know that {\em if we weight by the martingale $M_t$},
the weighted paths have the distribution of $SLE_{8/3}$ in
$D$ which is the same as the conformal image of $SLE_{8/3}$
in $\Half$ under $\Phi_D^{-1}$.  Since $\kappa \leq 4,$ we
know that these paths stay in $D$ and hence
\[    \lim_{\epsilon \rightarrow 0+} 
 \E[M_{\rho_\epsilon}; \rho_\epsilon < \infty]  = 0 , \]
and
\[  \Phi_D'(0)^{5/8} = \lim_{\epsilon \rightarrow 0+}
     \E [M_\infty; \rho_\epsilon = \infty]  =
   \E[1\{\gamma(0,\infty) \subset D\}]  = \Prob\{\gamma(0,\infty)
  \subset D\}. \]

\end{proof}

This generalizes to all $\kappa \leq 4$.

\begin{theorem} [Restriction]  If $\kappa =2/a \leq 4$ and $D \in
\domains$, then
 \[  
          \frac{d\mu_D(0,\infty)}{d\mu_\Half(0,\infty)}
           = M_\infty = 1\{\gamma(0,\infty) \subset D\}
    \exp\left\{-\c \int_0^\infty  \frac{aS\Phi_t(U_t)}{12} \, dt
  \right\} . \]
\end{theorem}

  For $\kappa \leq 8/3$ $(\c \leq 0$), 
the martingale $M_t$ is
uniformly bounded and the proof proceeds as the previous proof.
For $8/3 < \kappa \leq 4$, the martingale is not uniformly bounded,
but we can use an argument using the Girsanov theorem as above..
In the next section, we study $M_\infty$.

\section{(Brownian) boundary bubbles}

The Brownian bubble measure $\nu_\Half(x,x)$ is an
infinite measure on curves $\gamma:(0,t_\gamma) \rightarrow
\Half$ with $\gamma(0+) =\gamma(t_\gamma-) = x$.  There are
a number of ways of defining it.  One way is to consider Brownian
motion starting at $x + \epsilon i$ and conditioned to leave
$\Half$ at $x$.  Instead of thinking of this as a probability
measure, we consider it as a finite measure with total mass
$\pi \, H_\Half(x + \epsilon i,x)  =  1/\epsilon$. Here
$H_\Half$ denotes the Poisson kernel.  (Recall that
$H_\Half(x+\epsilon i , x) = (\pi \epsilon)^{-1} . $) We define
the Brownian bubble measure at $x$ by
\[               \nu_\Half(x) = 
\lim_{\epsilon \rightarrow
0+} {\pi}  
  \, \nu (x+\epsilon i,x)=
\lim_{\epsilon \rightarrow
0+} \frac{1}{\epsilon}  
  \, \nu^\#(x+\epsilon i,x) , \]
where $\nu^\#(x+\epsilon i,x)$ is the probability measure corresponding
to Brownian motion starting at $x + \epsilon i$ conditioned to
leave $\Half$ at $x$. 
If $D$ is a subdomain of $\Half$, then we let $\bubble(x,D)$ be
the $\nu_\Half(x,x)$ measure of all $\gamma$ with $\gamma(0,t_\gamma)
\not\in D$.  If $\dist(x,\Half \setminus D) > 0$, then 
$\bubble(x,D)$ is finite.  We let $\bubble(D) = \bubble(0,D)$.

The quantity $\Gamma(x,D)$ is expressed nicely in terms of
excursion measure.  If $D$ is a domain and $z,w$ are distinct points
in $\p D$ at which the boundary is smooth, the {\em excursion
Poisson kernel} (sometimes called the Dirichlet to Neumann map) is
defined by
\[            H_{\p D}(z,w) = \lim_{\epsilon \rightarrow 0+}
                  \epsilon^{-1} \, H_D(z + \epsilon {\bf n},w), \]
where ${\bf n}$ denotes the inward unit normal at $z$.  If $f:D
\rightarrow D'$ is a conformal transformation and the boundary
is sufficiently smooth, we have the scaling rule
\[            H_{\p D}(z,w) = |f'(z)| \, |f'(w)|
  \, H_{\p D'}(f(z),f(w)). \]
Note that this is the scaling rule from Lecture 1 with $b=1$.
If $D_1 \subset D$ and the boundaries are nice, we define
\[  \bubble_D(z,D_1) = \int_{\p D_1} H_{\p D_1}(z,w)
  \, H_D(w,z) \, |dw|. \]
We could also write this as the integral over $D \cap \p D_1$.
In this notation $\bubble(x,D) = \bubble_\Half(x,D)$.  The
definition of $\bubble_D(z,D_1)$ does not need smoothness of $D
\cap \p D_1$ --- the same integral can be expressed in terms
of Brownian excursion measure.  Using the scaling rule
for the (regular and excursion) Poisson kernel, we get the following
conformal covariance rule
\[   \bubble_D(z,D_1) = |f'(z)|^2 \, \bubble_{f(D)}(f(z),f(D_1)). \]

\begin{example}

Suppose $D = \Disk_+ = \{z \in \Half: |z| < 1\}$.  As
$\epsilon \rightarrow 0$, 
\[                 H_D(\epsilon i, e^{i\theta})
             \sim \frac{2\epsilon}{\pi} \, \sin \theta , \]
and hence
\[ \bubble(D) = \int_0^\pi H_{\p \Disk_+} (0,e^{i\theta})
\, H_{\Half}(e^{i\theta},0) = 
\int_0^\pi \frac{2}{\pi} \, \sin^2 \theta
  \, d\theta = 1. \]
The normalization of the bubble measure is chosen so
that $\bubble(\Disk_+) = 1$. 
\end{example}

  The next proposition (Exercise \ref{ex.bubble})
relates $\bubble$ to the
Schwarzian derivative.

\begin{proposition}  \label{feb27.prop2}
Suppose $D \subset \Half$ is a simply connected
domain with $\dist(0,\Half \setminus
D) > 0$.  Suppose $f: D \rightarrow \Half$
is a conformal transformation.  
Then,
\[            \bubble_\Half(0,D) = -\frac 16 \, Sf(0) , \]
where $S$ denotes Schwarzian derivative.
\end{proposition}

We can write the local martingale $M_t$ from the previous section
as
 \[  M_t = \left[\exp\left\{- 
 \int_0^t  {a\, \bubble(U_s,g_s(D))}  \, ds
  \right\}\right]^{-\c/2} \, \Phi_t'(U_t)^b, \]
and if $\kappa \leq 4$,
\[   M_\infty = 1\{\gamma(0,\infty) \subset D\}
  \, [e^{- \Theta}]^{-\c/2}, \]
where
\[   \Theta = \Theta(\gamma,D) = 
 a\int_0^\infty  {  \bubble(U_t,g_t(D))}  \, dt. \]
Note that $\Theta$ is a deterministic function of $\gamma$ and $D$.
The factor $a$ comes from the fact that $\gamma$ has been
parametrized so that $\hcap[\gamma(0,t] = at]$.  The value
of $\Theta$ does not depend on the paramertrization; we 
can write
\[   \Theta = \Theta(\gamma,D) = 
  \int_0^\infty  {  \bubble(U_t,g_t(D))}  \; d \hcap(\gamma(0,t]). \]
The very nice feature of $\Theta$ is conformal invariance.
More generally if $\gamma$ is a curve connecting boundary points
in $D$ and $D_1 \subset D$ we can define $\Theta_D(\gamma,D_1)$.

\begin{proposition}  If $D_1 \subset D$; $z,w$ distinct points
on $\p D$; and $\gamma$ a simple curve from $z,w$ in $D$, then
\[       \Theta_D(\gamma;D_1) = \Theta_{f(D)}(f\circ \gamma;
   f(D_1)) . \]
\end{proposition}

\begin{remark}  Because $\Theta$ is a conformal {\em invariant},
we can see that we no longer need to assume that $\p D$ is
smooth at $z$.
\end{remark}

\begin{remark}  We do not need $D$ or $D_1$ to be simply
connected in order to define $\Gamma_D(z,D_1)$
or $\Theta_D(\gamma;D_1)$.  However,
Proposition 
\ref{feb27.prop2} which relates the bubble measure to the
Schwarzian derivative does assume that the domain is simply
connected.
\end{remark}

\section{Brownian loop measure}

A {\em (rooted) loop} in $\C$ is a continuous function $\gamma:[0,t_\gamma]
\rightarrow \C$ with $\gamma(0) = \gamma(t_\gamma)$.  There
is a one-to-one correspondence between loops and ordered triples
\[           (z,s, \eta) , \]
where $z \in \C, s > 0$, and $\eta$ is a loop with
$t_{\eta} = 1$ and $\eta(0) = \eta(1)=0$.  The correspondence is
by translation and Brownian scaling,
\[    t_\gamma = s, \;\;\;\;\;
      \gamma(t) = z + s^{1/2} \, \eta(t/s),\;\;\;\;
    0\leq s \leq t_\gamma . \]
An {\em unrooted loop} is an equivalence class of rooted loops under the
equivalence $\tilde \gamma \sim \gamma$ if $\tilde \gamma(t) =
\gamma(t+r)$ for some $r \in \R$ (here addition is modulo $t_\gamma$).
In other words, an unrooted loop is a loop that forgets where its 
starting point is.

\begin{definition} $\;$

\begin{itemize}

\item The {\em (rooted) Brownian loop measure} is the measure on loops
given by putting the following measure on $(z,s,\eta)$:
\[            {\rm area } \times \frac{dt}{2 \pi t^2} \, 
                      \times \mbox{ Brownian bridge},\]
where Brownian bridge refers to the probability measure on Brownian
paths in $\R^2$ conditioned to return to the origin at time $1$.

\item  The {\em (unrooted) Brownian loop measure} is the measure
obtained from the rooted loop measure by forgetting the roots.

\end{itemize}

\end{definition}

\begin{remark}
Recall that the density for a two-dimensional Brownian motion
at time $t$ is $(2 \pi t)^{-1} \, e^{-|z|^2/2t}.$  Roughly speaking,
we can think of $(2 \pi t)^{-1}$ as the probability that a Brownian
motion starting at the origin is at the origin at time $t$.  We
think of the unrooted loop measure as giving measure about
$(2 \pi t)^{-1}$ to each unrooted loop of length $t$.  The rooted loop
measure then chooses a root for the loop by using the uniform distribution
on $[0,t]$; this gives the extra factor of $t^{-1}$ in the
definition.  These heuristics can be made precise by considering
random walk approximations to the loop measure.
\end{remark}

The unrooted loop measure turns out to be conformally invariant.
In fact, $\Theta_D(\gamma,D_1)$ is the measure of unrooted loops
in $D$ that intersect both $\gamma$ and $D_1$. 

\section{The measure $\mu_D(z,w)$ for $\kappa \leq 4$}

The ideas in this lecture can be used to define 
\[           \mu_D(z,w) = C(D;z,w) \, \mu_D^\#(z,w) , \]
for all simply connected domains $D$ and distinct boundary
points $z,w$ at which $\p D$ is locally analytic.  We allow
$w = \infty$.  In this case we say $\p D$ is locally analytic
at $\infty$ if $\p[f(D)]$ is locally analytic at $0$ where
$f(z') = 1/z'$.  Given any two such triples $(D,z,w)$,
$(D_1,z_1,w_1)$ there is a unique
conformal transformation, which we call the canonical
transformation,  
 $f:D \rightarrow D'$  with $f(z) = z_1, f(w) = w_1$ and
$|f'(w)| = 1$ with the appropriate interpretation of this
if  $w= \infty$  or
$w_1 = \infty$. (For example, if
$w= \infty, w_1 \neq \infty$, then as $z' \rightarrow
\infty$, $|f(z') - w_1| \sim |z'|^{-1}.$)
We define $C_\Half(0,\infty) = 1$ and $\mu_\Half(0,\infty)
= \mu_\Half^\#(0,\infty)$ to be the probability measure given
by $SLE_\kappa$.
 We then define
\[           \mu_D(z,w) = |f'(0)|^{-b} \, [f \circ
    \mu_\Half(0,\infty)] , \]
where $f: \Half \rightarrow D$ is the canonical transformation
with $f(0) = z, f(\infty) = w$.   In other words,
\[     C(D;z,w) = |f'(0)|^{-b}, \;\;\;\;
     \mu_\Half^\#(0,\infty) =
    f \circ  \mu_\Half^\#(0,\infty). \]

\begin{example}  The map $f(z) = z/(1-z)$ is the canonical
transformation from $\Half$ to $\Half$ with $f(0) = 0, f(1)
= \infty$.  Using this we see that $C(\Half;0,1) = 1$.
More generally,
\[              C(\Half;x_1,x_2) = |x_2 -x_1|^{-2b}. \]
\end{example}

\begin{example}  If $D \in \domains$, then $\Phi_D^{-1}$
is a canonical
transformation from $\Half$ to $D$.  Therefore,
\[             C(D;x,\infty) = \Phi_D'(x)^b. \]
\end{example} 

We get the following properties.

\begin{itemize}

\item {\bf Conformal covariance}.  If $f: D \rightarrow
f(D)$ is a conformal transformation, then
\begin{equation}  \label{may19.3}
           f \circ \mu_D(z,w) = |f'(z)|^b \, |f'(w)|^b
  \, \mu_{f(D)}(f(z),f(w)) 
\end{equation}
(assuming sufficient smoothness at the boundary points).   

\item {\bf Boundary perturbation}.
If        $D_1 \subset D$, $\p D, \p D_1$ agree
near boundary points $z,w$,      
then
\begin{equation}  \label{radon3}
      \frac{d\mu_{D_1}(z,w)}{d\mu_D(z,w)}\, (\gamma)
= 1\{\gamma \subset D_1\} \, e^{\c \Theta/2} , 
\end{equation}
where $\Theta$ is the measure of the set of Brownian loops
in $D$ that intersect both $\gamma$ and $D \setminus
D_1$.

\item  In particular, if $D_1 \subset D$, $\p D, \p D_1$ agree
near boundary points $z,w$,  and $f: D \rightarrow
f(D)$ is a conformal transformation, then
\[       \frac{d\mu_{D_1}(z,w)}{d\mu_D(z,w)}
            = \frac{d \mu_{f(D_1)}(f(z),f(w))}
                   {d\mu_{f(D)}(f(z),f(w))}. \]

\item  If        $D_1 \subset D$, $\p D, \p D_1$ agree
near boundary points $z,w$, then the probability measure
$\mu_{D_1}^\#(z,w)$ can be obtained from $\mu_D^\#(z,w)$ by
``weighting paths locally'' by $C(D_1;z,w)$, or equivalently
by ``weighting paths locally'' by
\[    \frac{d\mu_{D_1}(z,w)}{d\mu_D(z,w)} . \]

\end{itemize}

\begin{remark}  The last property is stated informally.   A
 precise statement for  $D = \Half, z = 0 , w=\infty$
was given in Section \ref{restmartsec}.  For other $D_1,D$,
we can first map $D$ to $\Half$ and use conformal invariance.
\end{remark}

\begin{remark}  Because the quantity
\[  \frac{d\mu_{D_1}(z,w)}{d\mu_D(z,w)} \]
is a conformal {\em invariant}, it is well defined even if $\p D$
is not smooth near $z,w$ (but still assuming that $\p D, \p D_1$
agree in neighborhoods of $z,w$). 
\end{remark}

\begin{remark}  If $D_1,D_2 \in \domains$ and $f:D_1 \rightarrow D_2$
is a conformal transformation with $f(\infty) = \infty$, we write
$f'(\infty) = r$ if $f(z) \sim z/r$ as $z \rightarrow \infty$.
In particular, if $f(z) = rz$, then $f'(\infty) = 1/r$.  Using this
convention,  \eqref{may19.3} holds for such transformations.  
\end{remark}

\section{Exercises for Lecture 4}

\begin{exercise} \label{ex.bubble}
 Prove Proposition \ref{feb27.prop2}.
\end{exercise}

\begin{exercise}  \label{may30.ex1}
Suppose $\gamma:(0,\infty) \rightarrow \Half$
is a simple curve with $\gamma(0+) = 0, \gamma(z) \rightarrow \infty$
as $z \rightarrow \infty$.  Let $g_t$ be the corresponding conformal
maps.  Fix $\kappa \leq 4$.
Show that if $0 < x < y$,
\[  C(\Half \setminus \gamma(0,t];x,y)
               = \frac{g_t'(x)^b \, g_t'(y)^b}{[g_t(y) - g_t(x)]^{2b}}, \]
and  
\[ C(\Half \setminus \gamma(0,\infty);x,y) = \lim_{t \rightarrow \infty}
   \frac{g_t'(x)^b \, g_t'(y)^b}{[g_t(y) - g_t(x)]^{2b}}. \]
\end{exercise}

\begin{exercise}  Suppose $\kappa \leq 4$.
Suppose $D$ is a bounded, simply connected
domain, and $A_1,A_2$ are disjoint, closed, analytic, subarcs of
$D$ larger than a single point.  Define the measure
\[   \mu_D(A_1,A_2) = \int_{A_1} \int_{A_2} \mu_D(z,w)
  \, |dw| \, |dz|, \]
where $\mu_D(z,w)$ denotes the $SLE_\kappa$ measure and
$ |d\cdot|$ denotes integration with respect to arc length.
We can write
\[     \mu_D(A_1,A_2) = C_D(A_1,A_2) \, \mu_D^\#(A_1,A_2), \]
where 
\[  C_D(A_1,A_2)=\int_{A_1} \int_{A_2} C(D;z,w)  \, |dw| \, |dz|, \]  is the total mass and $\mu_D^\#(A_1,A_2) $
is a probability measure.  Suppose $f:D \rightarrow D_1$
is a conformal transformation with $f(A_1), f(A_2)$ being
analytic arcs.
\begin{itemize}
\item  Convince yourself that the integral makes sense (i.e.,
there is no trouble integrating a ``measure-valued'' function).
\item Show that $0 < C_D(A_1,A_2) < \infty$.
\item  In the cases of self-avoiding walk, loop-erased walk, and
Ising interface, describe what $\mu_D(A_1,A_2)$ represents in
terms of limits of discrete models.
\item  Show that if $\kappa = 2$, then 
\[            f \circ \mu_D(A_1,A_2) = \mu_{D_1}(f(A_1),f(A_2)). \]
In particular, 
\begin{equation} \label{may23.1}
   f \circ \mu_D^\#(A_1,A_2) = \mu_{D_1}^\#(f(A_1),f(A_2)). 
\end{equation}
\item  Show that \eqref{may23.1} does not necessarily hold if
$\kappa \neq 2$.

\end{itemize}

\end{exercise}
\lecture{Radial and two-sided radial $SLE_\kappa$}

\section{Example: SAW II}  \label{may20sec}

We return to the example of the self-avoiding walk (SAW) from
Lecture 1.  We let $z$ be a boundary point and $w$ be an {\em interior}
point and consider the set of SAWs from $z$ to $w$.  We again
give the measure $e^{-\beta |\omega|}$ to each walk where $\beta$ is
the critical value and we let $Z_N(D;z,w)$ be the partition function
as before.  It is conjectured that
\[         Z_N(D;z,w) \sim C(D;z,w) \, N^{-b} \, N^{-\tilde b} , \]
where now we have a  {\em (one-sided)
interior scaling exponent} $\tilde b$.
If we multiply by $N^{b + \tilde b}$ and take a limit, we expect to
have a measure on simple paths from $z$ to $w$ (or $w$ to $z$).  The
model for the continuum limit of this is called {\em (one-sided)
radial $SLE$}.

\begin{figure}[htb]
\begin{center}
\epsfig{file=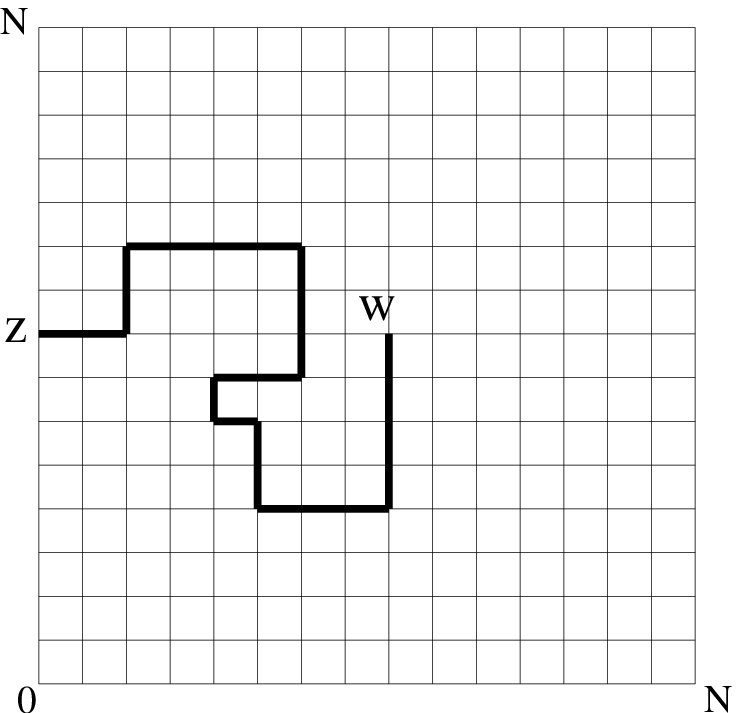}
\caption{Discrete approximation of one-sided radial}  
\end{center}
\end{figure}

We can also consider a case with two boundary points, $z,y$ and
one interior point $w$.  We look at the set of all SAWs $\omega$
from $z$
to $y$ that go through $w$.  Equivalently, we can look at the
set of pairs of SAWs $(\omega,\omega')$ where $\omega$ goes from
$w$ to $z$; $\omega'$ gives from $w$ to $y$; and $\omega \cap
\omega' = \{w\}$.  In this case, we expect the
partition function to scale like
\[     Z_N(D;z,y;w) \sim C(D;z,y;w) \, N^{-2b} \, N^{-\hat b} , \]
where $\hat b$ is a two-sided interior
scaling exponent.  By comparison
to $Z_N(D;z,y)$, we can see that $N^{-\hat b}$ should be comparable
to the probability that a SAW from $z$ to $y$ goes through $w$ which
we can conjecture to be about $N^{d}/N^2$ where $d$ is the fractal
dimension of the paths.  We therefore get
\[                     \hat b = 2-d . \]
Since this relation holds, we do not adopt the notation $\hat b$
but rather just refer to the exponent as $2-d$. 

Consider the marginal measure on $\omega$.  Then for any
$\omega$ from $w$ to $z$ the measure is
\[        Z_N(D \setminus \omega; y,w) . \]
In other words we can first choose $\omega$ using the one-sided
measure but then we weight this distribution by the measure of
walks $\omega'$ from $w$ to $y$ that avoid $\omega$.  We
will be able to look at the scaling limit of the measure
on $\omega$ in two different ways:
chordal $SLE$ from $z$ to $y$ conditioned to go through $w$ or
radial $SLE$ from $z$ to $w$ weighted by the total mass of paths
from $w$ to $y$ in $D \setminus \omega$.  (Both of these
interpretations must be considered in some kind of limit.)
If we fix $\omega$ and consider the probability measure on $\omega'$
obtained by conditioning, then this will be the same as the
probability measure given by the normalized measure on
$D \setminus \omega$.  In the scaling limit, the probability measure
on $\omega'$ given $\omega$ will be $\mu_{D \setminus \omega}^\#(w,y)$
(note that $w$ is a boundary point of $D \setminus \omega$).
 
\begin{figure}[htb]
\begin{center}
\epsfig{file=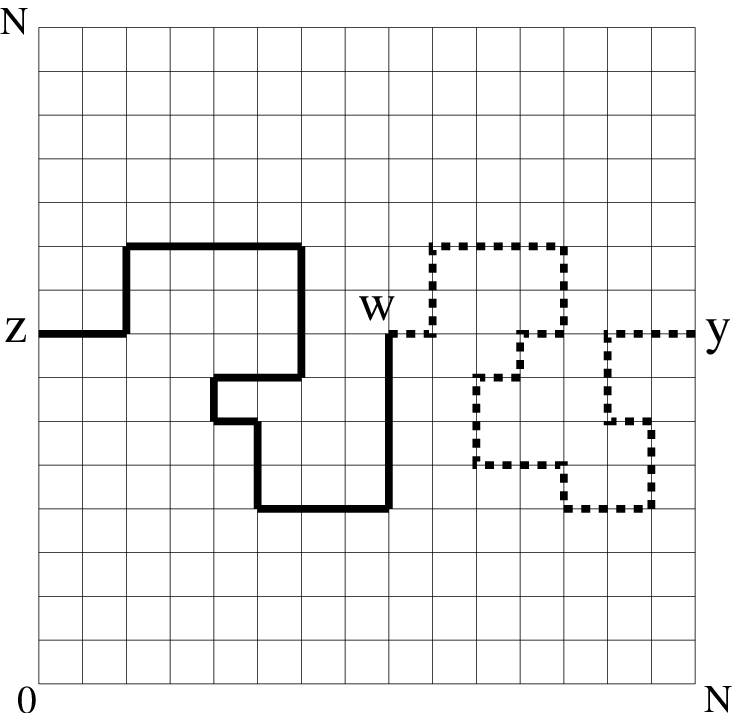}
\caption{Discrete approximation of two-sided radial}  
\end{center}
\end{figure}

Note that 
\[  \frac{Z_N(D;z,y;w)}{Z_N(D;z,w) \, Z_N(D;y,w)}
           \asymp   N^{-(\hat b - 2 \tilde b)}. \]
The left-hand side represents the probability that two
SAWs starting at the origin do not intersect.  If we recall
  from Lecture 1 that we expect this probability
to decay like $|\omega|^{-(\gamma - 1)}$, then we see that
we would expect this quantity to decay like $N^{-d(\gamma - 1)}$.
This gives the following relation between the exponents
\[          d(\gamma - 1) = \hat b - 2 \tilde b = 2-d - 2 \tilde b,
\;\;\;\;\;\;  \gamma = 2\nu(1 - \tilde b). \]

One can also consider the case for two interior points $w,w'$
for which
we would expect the scaling
\[             Z_N(D;w,w') \sim C(D;w,w') \, N^{-2 \tilde b}. \]
The corresponding probability measures are called {\em whole plane
$SLE$}.  We will not consider this case in these lectures. 

\section{Radial $SLE_\kappa$} 

Radial $SLE_\kappa$ is a measure 
\[ \mu_D(z,w) = C_D(z,w) \, \mu_D^\#(z,w) , \;\;\;\;
  z \in \p D, w \in D \]
  on paths $\gamma:(0,t_\gamma]
\rightarrow D$ with $\gamma(0+) = z$ and
$\gamma(t_\gamma) = w$. We will consider paths modulo
reparametrization, but we could also consider the paths
with parametrization.
  It satisfies
the following properties:
\begin{itemize}
\item  {\bf Domain Markov property  (for
$\mu_D^\#(z,w)$)}.  Given $\gamma(0,t]$, the
distribution   of the remainder of the path
is the same as $\mu^\#_{D \setminus \gamma(0,t]}(\gamma(t),
w)$
\item  {\bf Conformal covariance}.  If $f: D \rightarrow f(D)$
is a conformal transformation, then
\begin{equation}  \label{may15.1}
            C(D;z,w) = |f'(z)|^{b} \, |f'(w)|^{\tilde b}
  \, C(f(D);f(z),f(w)), 
\end{equation}
\[           f \circ \mu_D^\#(z,w) = \mu_{f(D)}^\#(f(z),f(w)). \]
\item {\bf Conformal invariance for boundary
perturbations  (for $\kappa \leq 4$)}.  If $w \in D_1 \subset D$;
$z \in \p D$; and $\p D_1,\p D$ agree near $z$; then $\mu_{D_1}(z,w)$
is absolutely continuous with respect to $\mu_D(z,w)$ and the
Radon-Nikodym derivative
\begin{equation}  \label{radon}
                \frac{d\mu_{D_1}(z,w)}{d\mu_D(z,w)} 
\end{equation}
is a conformal invariant.
\end{itemize}

\begin{remark}
Here we have introduced the interior scaling exponent
\[  \tilde b =  \frac{\kappa-2}{4}\, b. \]
We discuss below why this value is chosen --- essentially,
it is the value that makes the quantity on the
right hand side of   \eqref{jun1.1}
a local martingale.
  The total mass $C(D;z,w)$ is nonzero and finite
provided that $\p D$ is locally analytic at $z$.  One
does not need smoothness of the boundary at $z$ to define
the probability measures $\mu_D^\#(z,w)$ or the
Radon-Nikodym derivative \eqref{radon}.
\end{remark}

\begin{remark}
We will normalize so that $C(\Half;0,i) = 1$.  Once we
have done this, 
the scaling rule \eqref{may15.1} determines $C(D;z,w)$ for all
simply connected domains (assuming we have determined
$\tilde b$). The domain Markov property and conformal invariance
will determine $\mu^\#_\Half(0,i)$ and from this we get
$\mu_D^\#(z,w)$ for all simply connected $D$.
\end{remark}

\begin{example}  We compute $C(\Half;0,z)$.
 Let $x > 0$ and 
\[  f(z) = \frac{z(1 + x^2)}{x(x-z)}, \;\;\; f'(z) = -\frac{1 + x^2}
 {(z-x)^2} , \;\;\;\;
   f(i) =   - \frac 1x  + i  ,\]
which is a conformal transformation of $\Half$ onto $\Half$ with
$f(0) = 0$.  Note that
\[           |f'(i)| = 
  1, \;\;\;\; |f'(0)|^{-1}  = 
\frac {x^2}{x^2 + 1}  
 = \pi\, H_\Half\left(- \frac 1x +i,0\right) , \]
where $H_\Half$ denotes the Poisson kernel.  Hence we get
\[          C(\Half; 0, y(x+i)) =
       y^{- \tilde b} \,y^{-b} \, \left[ \pi\, H_\Half\left( x +i,0
  \right)\right]^b
= y^{-\tilde b} \,\left[\pi H_\Half(y(x+i),0) \right]^b  . \]
\end{example}

\begin{example}  We choose (somewhat arbitarily)
\[          C(\Half;\infty,i) =1. \]
The map $F(z) = rz$ satisfies $F'(\infty) = 1/r$, so we get
\[        C(\Half;\infty,x+iy) = C(\Half;\infty,iy) = y^{b - \tilde
  b}, \]
If $w \in D \in \domains$, then $\Phi_D'(\infty)
=1$ and
\[  C(D;\infty,w) =  |\Phi_D'(w)|^{\tilde b} \,C(\Half;\infty,\Phi_D(w)).\]
In particular, if $\gamma$ is a path and $t < T_w$, then
\[  C(\Half \setminus \gamma(0,t]; \infty,w)
         = |g_t'(w)|^{\tilde b} \, Y_t^{b - \tilde b}
   = \distsub_t^{-\tilde b} \, Y_t^b . \]
\end{example}

If $g_t$ is chordal $SLE_\kappa$ from $0$ to $\infty$, then (see
\eqref{may15.2}), 
\begin{equation}  \label{jun1.1}
     M_t = |g_t'(z)|^{\tilde b} \, C(\Half;U_t,g_t(z))
   =|g_t'(z)|^{\tilde b} \, C(\Half;0,Z_t) =
   \distsub_t^{-\tilde b} \,\left[\pi H_\Half (Z_t,0) \right]^b , 
\end{equation}
is a local martingale for $t \leq T_z$ satisfying
\[          dM_t = (1-3a) \, \frac{X_t}{X_t^2 + Y_t^2}
  \, M_t \, dB_t. \]
 We can consider this property as the defining property for
the exponent $\tilde b$. 
If we use the
Girsanov theorem and weight the paths by $M_t$ then under the
weighted measure $U_t = -B_t$ satisfies
\[             dU_t =  (3a-1)  \, 
\frac{X_t}{X_t^2 + Y_t^2} \, dt - dW_t , \]
where $W_t$ is a Brownian motion in the new measure.   This leads
to one definition for the probability measure
$\mu_\Half(0,w)$.

\begin{definition}  Let $w \in \Half$ and
suppose $g_t$ is the solution to the
Loewner equation \eqref{chordalspec} where $U_t$ satisfies
\[    dU_t =  (3a-1)  \, 
\frac{X_t}{X_t^2 + Y_t^2} \, dt - dW_t , \]
where $X_t = X_t(w), Y_t = Y_t(w)$.  Then the corresponding
curve $\gamma(t), 0 < t < T_w$ is {\em radial} $SLE_\kappa$
from $0$ to $w$ stopped at time $T_w$.
\end{definition}

Note that for radial $SLE_\kappa$,
\[   dX_t = \frac{(1-2a) \, X_t}{X_t^2 + Y_t^2} \, dt
          + dW_t, \]
\[  d \Theta_t = \frac{a X_tY_t}{(X_t^2 + Y_t^2)^2} \, dt
              - \frac{Y_t}{X_t^2 + Y_t^2} \, dW_t. \]

\begin{proposition}  $\;$

\begin{itemize}
\item  If $\kappa \leq 4$,  then $T_w <\infty$ and $\gamma(T_w-)
  = w$.
\item  If $4 < \kappa < 8$, then $T_w < \infty$ and $\gamma(T_w-)
\neq w$.
\item  If $\kappa \geq 8$, then $T_w = \infty$.
\end{itemize}
\end{proposition}

\begin{proof}
We consider a new parametrization for which the conformal
radius $\distsub_t$ decreases deterministically.  Assume for
ease that $w=x+i$ and let $\rho(t)$ be defined by
$\distsub_{\rho(t)} = e^{-2at}$.  In other words,
\[            \dot \rho(t) = \frac{(X_{\rho(t)}^2
   + Y_{\rho(t)}^2)^2}{Y_{\rho(t)}^2} . \]
This is valid up to the time that $Y_{\rho(t)} = 0$.
Recall that $\distsub_t
\asymp \dist[w,\gamma(0,t] \cup \R]$. Then $\tilde \Theta_t
= \Theta_{\rho(t)}$ satisfies
\begin{equation}  \label{theta}
         d \tilde \Theta_t = a \, \cot \tilde \Theta_t
  \, dt + d \tilde W_t , 
\end{equation}
where $\tilde W_t$ is a standard Brownian motion.  By comparison
with a Bessel process we see that if $a \geq 1/2$, then $
\sin \tilde \Theta_t > 0$ for all $t$, but  if $a < 1/2$, then
$\sin \tilde \Theta_t$ reaches zero in finite time.
\end{proof}

\begin{remark} $\;$ 

\begin{itemize}  

\item  Our definition shows immediately that radial $SLE_\kappa$
up to time $T_w$ is absolutely continuous with respect to
chordal $SLE_\kappa$.  (To be more accurate, we need to stop the process
slightly before $T_w$).

\item  If $\kappa = 2/a = 6$, then  radial $SLE_\kappa$
has the {\em same} distribution as chordal $SLE_\kappa$ up to
time $T_w$.  This is another version of the locality property
for $SLE_6$. 

\item  For $\kappa  > 4$, this definition only defines radial
$SLE_\kappa$ up to time $T_w$. In order to define the
measure $\mu_\Half^\#(0,i)$
we can consider  $\hat g_t$,
  a   conformal transformation of $\Half \setminus
\gamma(0,t]$ onto $\Half$ with $\hat g_t(i) = i$.  This is the
basis of the original defintion which we describe in the
next section. There the definition is done for
$\mu_\Disk^\#(z,0)$.

\end{itemize}

\end{remark}

\section{Another definition}  \label{originalsec}

Here we give an alternative  
definition of radial $SLE_\kappa$ in simply connected domains.
We will define $\mu_\Disk^\#(z,0)$ for $z \in \p \Disk$. Given
$\gamma(0,t]$, let $\tilde g_t$ denote the unique
conformal transformation of $\Disk \setminus \gamma(0,t]$ onto
$\Disk$ with $\tilde g_t(0) = 0, \tilde g_t'(0) > 0$. If we parametrize
the curve so that $\log \tilde g_t'(0) =  t$, then 
 $\tilde g_t$ 
satisfies
\[                    \p_t \tilde g_t(z) =  
  \tilde g_t(z) \, \frac{e^{i2U_t} + \tilde g_t(z)}
                {e^{i 2U_t} - \tilde g_t(z)} , \;\;\;\;
  \tilde g_0(z) = z , \]
  where $\tilde g_t(\gamma(t))
= e^{i 2U_t}$  If we choose $2U_t = B_{\kappa t}$ where 
  $B_t$ is a standard Brownian motion, then 
$\gamma$ has the distribution of $\mu_\Disk^\#(z,0).$
$\mu_D^\#(z,w)$ for other simply connected domains is defined by conformal invariance.  The equivalence of the definitions of 
radial $SLE_\kappa$
 can be checked in a straightforward way by studying the
image of $SLE$ under conformal
transformations.  In studying the equation above it is
often useful to consider the logarithm: suppose $\tilde g_t(e^{2i\theta})
 = e^{2i\tilde h_t(\theta)}. $  (Since $\tilde g_t$ vanishes at the
origin, there are techinical issues in taking the logarithm.  
If we stay in 
a simply connected subdomain of $\Disk$ that avoids the origin, e.g.,
a neighborhood of a boundary point, there is no problem.)  Then
$h_t$ satisfies
\[          \p_t\tilde h_t(\theta) = \frac 12 \, 
\cot\left( h_t(\theta) - \frac 12 \,
    B_{\kappa t}\right). \]    
If $h_t =  \tilde h_{4t/\kappa}$, then this becomes
\[          \p_t h_t(\theta) = a\, 
\cot\left(\tilde h_t(\theta) -   \,
    \tilde W_{ t}\right), \]   
where $\tilde W_t$ is a standard Brownian motion.  In other words
$\Psi_t := h_t(\theta) - \tilde W_t$ satisfies
\[           d\Psi_t = a \, \cot \Psi_t \, dt + dW_t,\]
where $W_t =  - \tilde W_t$ is a standard Brownian motion.    Note
that this is the same equation as \eqref{theta}. If we let
$g_t = \tilde g_{4t/\kappa}$, then
\[ |g_t'(e^{i2\theta})| =  \exp\left\{-\int_0^t \frac{a \, ds}
     {\sin^2 \Psi_s} \right\}. \]
in this parametrization $g_t'(0) = e^{4t/\kappa} = e^{2at}$. 
 
\section{Radial $SLE_\kappa$ in a smaller domain}

The computation starts getting a little messy, but we can
also do radial $SLE$ in a smaller domain.

\begin{proposition}
Let $D \in \domains$ and $w \in D$. 
Let 
\[ K_t = C(D_t;U_t,g_t(z)) = \Phi_t'(U_t)^b \,
   |\Phi_t'(g_t(z))|^{\tilde b}  
  \,  C(\Half;0,L_t), \]
where $L_t = \Phi_t(g_t(z)) - \Phi_t(U_t)$. 
Then,
\[ M_t = \exp\left\{ - \int_0^t
 \frac{a\c}{12} \, S\Phi_t(U_s) \, ds\right\} \, K_t ,\]
is a local martingale.
\end{proposition}

If $\kappa \leq 4$, then
\eqref{radon3} holds for radial $SLE_\kappa$.  In particular,
for $\kappa = 8/3$, radial $SLE_\kappa$ satisfies the
restriction property. 

\begin{proposition} Suppose $w \in D_1 \subset D$ where $D_1,D$
are simply connected.  Suppose $z \in \p D$; $\p D$ is smooth near
$z$; and $D_1,D$ agree near $z$.  Then if $\gamma$ has the
distribution of $\mu^\#_D(z,w)$, then the probability that
$\gamma(0,t_\gamma] \subset D_1$ is $|F'(z)|^{5/8} \,
 |F'(w)|^{5/48}$ where $F$ is the conformal transformation of
$D_1$ onto $D$ fixing $z,w$.
\end{proposition}

\section{Two-sided radial}

Here we let $\kappa < 8$
and introduce a process that can be called either
{\em two-sided radial} $SLE_\kappa$  or  {\em chordal} $SLE_\kappa$
{\em conditioned to go through a point}.  This is  a measure
\[   \mu_D(z_1,z_2;w) =  C(D;z_1,z_2;w) \,   \mu_D^\#(z_1,z_2;w), \]
where $D$ is a domain; $z_1,z_2$ are distinct points in $\p D$;
and $w \in D$.  We define
\[    C(\Half;0,\infty;w) =  G(w) , \]
where $G(y(x+i)) = y^{d-2}\, (x^2 + 1)^{\frac 12 -2 a}$ is the
Green's function for chordal $SLE_\kappa$ as introduced in
Lecture 3.  For other simply connected domains, we define
\[         C(D;z_1,z_2;w) = 
 |f'(0)|^{-b} \, |f'(w')|^{-\hat b} \, G(w')=
 |f'(0)|^{-b} \, |f'(w')|^{d-2} \, G(w'). \]
where $f: \Half \rightarrow D$ is the canonical transformation
with $f(0) = z_1, f(\infty) = z_2$ and $w' = f^{-1}(w) . $ 
It satisfies the following conformal covariance relations.

\begin{itemize}
\item  {\bf Conformal covariance}.  If $f: D \rightarrow f(D)$
is a conformal transformation, then
\begin{equation}  \label{may19.1}
            C(D;z_1,z_2;w) = |f'(z_1)|^{b} \, |f'(z_2)|^{  b}
\, |f'(w)|^{2-d} 
  \, C(f(D);f(z_1),f(z_2); f(w)), 
\end{equation}
\[           f \circ \mu_D^\#(z_1,z_2;w) = \mu_{f(D)}^\#(f(z_1),
f(z_2);f(w)). \]
\item {\bf Conformal invariance for boundary
perturbations} (for $\kappa \leq 4$).  If $w \in D_1 \subset D$;
$z \in \p D$; and $\p D_1,\p D$ agree near $z_1,z_2$; then 
$\mu_{D_1}(z_1,z_2;w)$
is absolutely continuous with respect to $\mu_D(z_1,z_2;w)$ and the
Radon-Nikodym derivative
\begin{equation}  \label{radon2}
                \frac{d\mu_{D_1}(z_1,z_2;w)}{d\mu_D(z_1,z_2;w)} 
\end{equation}
is a conformal invariant.
\end{itemize}

There are two different ways to think of the probability
measure $\mu_D^\#(z_1,z_2;w)$.  We illustrate this with $D = \Half,
z_1=0, z_2=\infty$.

\subsection{Chordal weighted by $G$}

Recall that if $z \in \Half$ and $M_t = M_t(z) =
  |g_t'(z)|^{2-d} \, G(Z_t), $ then $M_t$ is a local martingale
satisfying
\[    dM_t = (1-4a) \, \frac{X_t}{X_{t}^2 + Y_t^2} \, M_t \, dB_t. \]
Hence if we weight the paths by the martingale, then 
\[   dB_t = (1-4a) \, \frac{X_t}{X_{t}^2 + Y_t^2} \, dt + d \hat B_t, \]
where $\hat B_t$ denotes a Brownian motion in the new measure.
In other words,
\[   dX_t = (1-3a) \, \frac{X_t}{X_{t}^2 + Y_t^2} \, dt + d \hat B_t,
 \]
\[   d\Theta_t = \frac{2aX_tY_t}{(X_t^2 + Y_t^2)^2} \, dt 
   - \frac{Y_t}{X_t^2 + Y_t^2} \, d \hat B_t. \]
If we do the time reparametrization as in the previous section, we
can write
\begin{equation}  \label{may21.1}
      d \tilde \Theta_t = 2a \cot \tilde \Theta_t \, dt
   + d \hat W_t , 
\end{equation}
where $\hat W_t$ is a standard Brownian motion (in the weighted
measure). 

\subsection{Radial weighted by radial weights}

We will now show that we can consider two-sided radial $SLE_\kappa$
from $0$ to $w$ in $\Half$ as (one-sided) radial $SLE_\kappa$
weighted by $C(\Half \setminus \gamma(0,t];\infty,w)$. For
ease, let us assume $w=i$.  Let us start
with radial $SLE_\kappa$ parametrized by $\distsub_t$ and 
recall that 
\[  C(\Half \setminus \gamma(0,\rho(t)];\infty,w)
                = \distsub_{\rho(t)}^{-\tilde b} \,  Y_{\rho(t)}^b
            = e^{-2a\tilde b} \,  \tilde Y_t^b. \]
Note that
\[  \p_t \tilde Y_t  =  - a  \, \csc^2 \, \tilde \Theta_t \, \tilde Y_t , \]
i.e.,
\[        J_t :=  \tilde Y_t^b= \exp\left\{-ab \int_0^t \frac{ds}
{\sin^2 \tilde \Theta_s}
  \right\}. \]
Recall that for radial $SLE_\kappa$, $\tilde \Theta_t$ satisfies
\eqref{theta}.  Let us now assume that there is a function
$\phi(\theta)$ such that
\[ \phi(\theta) = \lim_{t \rightarrow \infty} \, e^{\lambda t}
  \, \E[J_t \mid \tilde \Theta_0 = \theta]. \]
Then $e^{\lambda t} \, \phi(\tilde \Theta_t)\, J_t$ is a martingale.
Using It\^o's formula, we get the equation
\[   
  {\phi''(\theta)}  + 2 a \, 
              \phi'(\theta) \, \cot \theta  +
  \left(2\lambda - {2ab} \csc^2 \theta \right) \,
  \phi(\theta) = 0 , \]
which gives $\phi(\theta) = e^{3at/2} \, \sin^a \theta$.
Therefore, $M_t = e^{3at/2} \, \sin^a \tilde \Theta_t$
satisfies
\[           dM_t = a \, \cot \tilde \Theta_t \, dW_t. \]
If we do Girsanov and weight by $M_t$, then $\tilde \Theta_t$
satisfies
\[            d \tilde \Theta_t = 2a \,  \cot \tilde \Theta_t 
 \, dt + d \tilde W_t , \]
where $\tilde W_t$ is a standard Brownian motion in the new measure.
 Note that this is the same as \eqref{may21.1}.

\section{Exercises for Lecture 5}

\begin{exercise}  Redo Section \ref{may20sec} for  
  loop-erased walk.
\end{exercise}

\begin{exercise}  \label{exer.may19}
Verify \eqref{may19.1}.
\end{exercise}

\lecture{Intersection exponents}

In this lecture we will compute the intersection exponents
for $SLE_\kappa, \kappa \leq 4$.  The particular case
of $\kappa = 8/3$ gives the Brownian motion intersection
exponents.  We fix $a = 2/\kappa \geq 1/2$ and   constants 
 depend on $a$.

\section{One-sided}

We assume that $g_t$ satisfies \eqref{chordalspec}  
where $U_t = - B_t$ is a standard Brownian motion. If $z \in
\Half$ we write $Z_t = Z_{t,z} = g_t(z) - U_t$ and we write
$Z_t = X_t + i Y_t$.  Recall that 
\[   dX_t = \frac{aX_t}{X_t^2 + Y_t^2} \, dt + dB_t, \;\;\;\;\;
   dY_t = - \frac{aY_t}{X_t^2 + Y_t^2} \, dt. \]
\begin{equation}  \label{oct5.6.complex}
         d|g_t'(z)| = |g_t'(z)| \, \frac{a(Y_t^2 - X_t^2)}{(X_t^2 + Y_t^2)^2}
  \, dt. 
\end{equation}
We will consider the case $z = x \in \R\setminus
\{0\}, t \leq
T_x$,  for which 
$g_t'(x)  \in (0,1]$, 
   \begin{equation}  \label{oct5.6}
       dX_t = \frac{a}{X_t} \, dt + dB_t,\;\;\;\;
   \p_t g_t'(x)  =   -\frac{a\, g_t'(x)}{X_t^2}, \;\;\;\;
   g_t'(x) = \exp\left\{-a\int_0^t \frac{ds}{X_s^2} \right\} 
   . 
\end{equation}
Define
\[ \lambda_0 = \lambda_0(a) = -\frac{(2a-1)^2}{8a}, \]
  \begin{equation}  \label{oct5.5}
   q(\lambda) = q_{+}(\lambda) = q(\lambda;a) =
  \frac{(1-2a) + \sqrt{(2a-1)^2 + 8a\lambda}}{2}, \;\;\;\;
\lambda \geq \lambda_0,
\end{equation}
\[   q_-(\lambda) = q_-(\lambda;a) =
  \frac{(1-2a) - \sqrt{(2a-1)^2 + 8a\lambda}}{2}, \;\;\;\;
\lambda \geq \lambda_0. \]
Note that $q = q_\pm(\lambda)$ is a solution to the
quadratic equation

\begin{equation}  \label{oct5.5.1}
q^2 + (2a-1) \, q - 2a\lambda = 0 , 
\end{equation}
and  that
\begin{equation}  \label{oct5.5.2}
 q(\lambda_1) + q(\lambda_2) = q\left(\lambda_1 + \lambda_2
   + \frac{q(\lambda_1)q(\lambda_2)}a \right). 
\end{equation}

\begin{proposition}
Suppose $x > 0,$ $\lambda \geq \lambda_0$,
  $q = q(\lambda)$.
Then
$           M_t := X_t^q \,  g_t'(x) ^\lambda $
is a   martingale satisfying
\begin{equation}  \label{may21.6}
   dM_t = M_t \, \frac{q}{X_t} \, dB_t. 
\end{equation}
If $q = q_-(\lambda)$, the same holds with ``local martingale''
replacing ``martingale''.
\end{proposition}

\begin{proof} 
It\^o's formula gives
\[   dX_t^q = X_t^q \, \left[\frac{qa + \frac 12 q (q-1)}{X_t^2}
  \, dt + \frac q{X_t} \, dB_t \right], \]
and hence the product rule and (\ref{oct5.6}) give
\[    dM_t = M_t \, \left[ \frac{qa + \frac 12 q (q-1)
  - a \lambda}{X_t^2} \, dt + \frac{q}{X_t} \, dB_t \right] . \]
  By (\ref{oct5.5.1}), we see
that  $M_t$ is a local
martingale. 
If we apply Girsanov and weight the paths by $M_t$,  
the paths $B_t$ satisfy
\[        dB_t = \frac{q}{X_t} \, dt + dW_t, \]
where $W$ is a standard Brownian motion in the new weighting.
Hence $X_t$ satisfies
\begin{equation}  \label{kapparho}
            dX_t = \frac{a+q}{X_t} \, dt + dW_t. 
\end{equation}
If $\lambda \geq \lambda_0$ then
$q \geq  (1/2) - a$ and hence $a + q \geq 1/2$.
Therefore, the weighted paths never reach the origin.  Also,
since the weighted paths follow a Bessel process, one
can use the criterion from Section \ref{girsanovsec} to see that $M_t$
is actually a martingale.
\end{proof}

If we consider the pair of processes $U_t = - B_t$ and $K_t =
  g_t(x)$, then the pair of processes $(U_t,K_t)$ satisfy
the simple system
\begin{equation}  \label{kapparho2}
  dU_t =  \frac{q}{U_t-K_t} \, dt - dW_t, \;\;\;\;\;
            \p_tK_t = \frac{a}{K_t - U_t}. 
\end{equation}
This is an example of an $SLE(\kappa,\rho)$ process where 
$\rho = q\kappa$.  (The parameter $\rho$ is the coefficient of
the drift of $U_t$ if the process is parametrized so that
$\hcap(\gamma(0,t]) = 2t$.  This is probably  not the best
way to parametrize these processes.)

\begin{example}  Let $x > 0$, $\lambda = b, q = q(\lambda)
 = a$.  Note that
$C(\Half \setminus \gamma(0,t]; x,\infty) = g_t'(x)^{b}$.
Then \eqref{kapparho} becomes
\[   dX_t = \frac{2a}{X_t} \, dt + dW_t, \]
and \eqref{kapparho2} becomes
\begin{equation}  \label{kapparho3}
 dU_t =  \frac{a}{U_t-K_t} \, dt - dW_t, \;\;\;\;\;
            \p_tK_t = \frac{a}{K_t - U_t}. 
\end{equation}
This is the $SLE(\kappa,\rho)$ process with $\rho = a \kappa
= 2$. It is also the measure for {\em two-sided chordal $SLE_\kappa$}
from $(0,x)$ to $\infty$ in $\Half$.  We can let $x \rightarrow 0$ and
obtain {\em two-sided chordal $SLE_\kappa$}
from $0$ to $\infty$ in $\Half$.  If we are only interested
in the distribution
of  $X_t = K_t - U_t$ we can replace \eqref{kapparho3} with
\[  dU_t =  \frac{a}{U_t-K_t} \, dt - \sqrt{r(t)} \, dW_t^1, \;\;\;\;\;
            d K_t = \frac{a}{K_t - U_t} \, dt - \sqrt{1 - r(t)}
  \, dW_t^2 , \]
where $ 0 \leq r(t) \leq 1$ and $W^1,W^2$ are independent standard
Brownian motions.  The value of $r(t)$ determines the ratio
of growth rates of the two paths. 
\end{example}

\begin{example}  Consider $\mu_\Half(0,x)
 = x^{-2b} \, \mu_\Half^\#(0,x).$  We can obtain
the  probability
measure $\mu_\Half^\#(0,x)$ from 
  $SLE_\kappa$ from $0$ to $\infty$
 by weighting by $X_t^{-2b}$, i.e.,
by choosing $q_- = -2b$. Then
\eqref{kapparho} becomes
\[            dX_t = \frac{1-2a}{X_t} \, dt + dW_t , \]
and
\eqref{kapparho2} becomes
\[ dU_t =  \frac{-2b}{U_t-K_t} \, dt - dW_t, \;\;\;\;\;
            \p_tK_t = \frac{a}{K_t - U_t}. \] 
This is $SLE(\kappa,\rho)$ with
\[          \rho =  -2b \kappa = \kappa - 6. \]
In this case the local martingale $M_t$ is not a martingale.
This can be see by noting that the amount of time (in $\hcap$
from infinity) to go from $0$ to $x$ is finite, i.e., there
is explosion in finite time.
\end{example}

\begin{proposition} \label{prop.may23} If $\lambda \geq \lambda_0$ and
$q = q(\lambda)$, there  is a constant $c$ such that
for all $t,x> 0$,
\[   \E[g_{tx^2}'(x)^\lambda] = \E[g_t'(1)^\lambda] \sim 
                c \,   t^{-q/2} . \]
\end{proposition}

\begin{proof} (Sketch)
 The equality holds by scaling so we may assume $x=1$. 
This is a fact about Bessel processes; we are computing
$  \E [J_t]$ where
\[            J_t = \exp\left\{-a\lambda \int_0^t 
  \frac{ds}{X_s^2} \right\} , \]
and $X_t$ is a Bessel process satisfying
\[  dX_t = \frac a{X_t} \, dt + dB_t. \]
We know that $M_t = J_t \, X_t^q$ is a martingale satisfying
\eqref{may21.6}.  If we weight the paths by $M_t$, then the
weighted paths satisfy \eqref{kapparho} where $W_t$ is
a standard Brownian motion in the new measure. 
In particular,
\[ \E[J_t] = \, \E[M_t \, X_t^{-q}] = M_0 \,\tilde  \E[X_t^{-q}], \]
where $\tilde \E$ denotes expectations in the new measure
$\tilde \Prob$.
In the 
$\tilde \Prob$ measure, $X_t$ satisfies
\eqref{kapparho}, and hence has a limiting distribution 
 proportional to
  $ r^{2(a+q)} \, e^{-r^2/2}$.   The constant $c$ is the
expectation of $X^{-q}$ with respect to this limit measure.
\end{proof}

\begin{remark}
We will interpret $q$ as an ``intersection exponent'' or
``crossing exponent'' for
$SLE_\kappa$.   Suppose
$L > 0$ is large (we will take asymptotics as $L \rightarrow
\infty$) and $\gamma$ is an $SLE_\kappa$ curve from $0$
to $\infty$.  Let us consider (see Exercise \ref{may30.ex1})
\[     
  C(\Half \setminus \gamma(0,\infty); 1,L) 
                 = \lim_{t \rightarrow \infty} 
 \frac{g_t'(1)^b\, g_t'(L)^b}
    {(g_t(L) - g_t(1))^{2b}} . \]
Since the diameter of $\gamma(0,t]$ is of order $\sqrt t$, we can
expect that
\[ C(\Half \setminus \gamma(0,\infty); 1,L) 
          \asymp  \frac{g_{L^2}'(1)^b\, g_{L^2}'(L)^b}
  {(g_{L^2}(L) - g_{L^2}(1))^{2b}} 
             \asymp    L^{-2b} \,   g_{L^2}'(1)^b , \]
\[   \E\left[ C(\Half \setminus \gamma(0,\infty); 1,L)^r\right]
          \asymp \E\left[\frac{g_{L^2}'(1)^{br}\, g_{L^2}'(L)^{br}}
    {(g_{L^2}(L) - g_{L^2}(1))^{2br}}\right]
             \asymp    L^{-2br} \, \E[  g_{L^2}'(1)^{br}]. \]
Or, in other words,
\begin{equation}  \label{may30.1} 
  {\E\left[ C(\Half \setminus \gamma(0,\infty); 1,L)^r\right] }
        \asymp  L^{ -q(br)-2b(r-1)}. 
\end{equation}

\end{remark}

\section{Two-sided}  \label{intersec}

We will now do a similar result for two processes.

\begin{lemma}  Suppose $y < 0 < x$,
$X_t = g_t(x) - U_t, \tilde X_t = U_t - g_t(y)$.
Let $\lambda_1,\lambda_2 \geq \lambda_0$ 
and let $q = q(\lambda_1), \tilde q = q(\lambda_2), r
= q\tilde q/a$. 
 Then
\[    M_t = X_t^{q} \, \tilde X_t^{\tilde q}
   \, (X_t + \tilde X_t)^r \,
   g_t'(x)^{\lambda_1} \, g_t'(y)^{\lambda_2} \]
is a martingale.
\end{lemma}

\begin{proof}
Recalling that
\[  dX_t = \frac{a}{X_t} \, dt + dB_t, \;\;\;\;
              d\tilde X_t = \frac a {\tilde X_t} \, dt - dB_t, \]
we can use It\^o's formula and the product rule to show that
\[  dM_t = M_t \,\left[ K_t \,dt  
 + \left(\frac{q}{X_t} - \frac {\tilde q}
{\tilde X_t} \right) \, dB_t \right], \]
where
\[ K_t =  \frac{qa + \frac 12 q(q-1) -
  a \lambda_1}{X_t^2} + \frac{\tilde q a + \frac 12 \tilde q(
  \tilde q -1) - a\lambda_2}
   {\tilde X_t^2}  
+\frac{ar-\tilde q q}{X_t\, \tilde X_t} 
      = 0 , \]
for our choice of $q,\tilde q,r$.
If we weight the paths by $M_t$, we get  
\[            dB_t = \left(\frac{q}{X_t} - \frac{\tilde q}
  {\tilde X_t} \right) \, dt + dW_t. \]
   Therefore,
\[     dX_t = \left(\frac{a+q}{X_t} - \frac{\tilde q}
  {\tilde X_t}\right) \, dt + dW_t,\]
\[  d \tilde X_t  = \left
(\frac{a+\tilde q}{\tilde X_t} - \frac{q}
  { X_t}\right) \, dt - dW_t. \]
By examining these coupled Bessel-like processes we can
see that $M_t$ is a martingale.
\end{proof}

This lemma and (\ref{oct5.5.2})
go a long way to proving an important estimate.  
For large $t$, $X_t \approx t^{1/2},
\tilde X_t \approx t^{1/2}$.  Since $\E[M_t] = M_0$,
we get
\[  \E[ g_t'(x)^{\lambda_1} \, g_t'(y)^{\lambda_2}]
  \asymp t^{-q(\lambda_1  ,\lambda_2)/2} , \]
where
\[   q(\lambda_1,\lambda_2) = q(\lambda_1) + q(\lambda_2)
   + \frac{q(\lambda_1)\,q(\lambda_2)}{a}. \]
  In fact, one can show that for each $x > 0$ there is a
$c_x >0$ such that
\[   \lim_{t \rightarrow \infty}
    t^{q(\lambda_1,\lambda_2)/2} \, \E[g_t'(x)^{\lambda_1}
       \, g_t'(-1)^{\lambda_2}] = c_x. \]

We will now define the chordal crossing exponents
$\tilde \xi = \tilde \xi_\kappa$
\begin{itemize}
\item   $\tx(\lambda) = \lambda.$
\item \[ \tx(\lambda_1,\lambda_2) = \lambda_1 + \lambda_2
   + \frac{q(\lambda_1) \, q(\lambda_2)}{a} . \]
Note that
\[    \tx(\lambda_1,\lambda_2) = q^{-1}
  \left(q(\tx(\lambda_1))
            +q(\tx(\lambda_2))\right) = q^{-1}
  \left(q(\lambda_1) + q(\lambda_2)\right).\]
\item More generally we define
\[   \tx(\lambda_1,\lambda_2,\ldots,\lambda_n)
  = q^{-1} \left(q(\lambda_1) + \cdots + q(\lambda_n)
  \right) . \]
\end{itemize}
Clearly $\tx$ is a symmetric function and one can check
that it satisfies the ``cascade relation''
\[    \tx(\lambda_1,\ldots,\lambda_n) =
  \tx(\tx(\lambda_1,\ldots,\lambda_k),\tx(\lambda_{k+1},
  \ldots,\lambda_n)). \]
If $\tx_n = \tx(\underbrace{b,\cdots,b}_n)$, then
\[   \tx_{n+1} = \tx(b,\tx_n) = \tx_n + b + q(\tx_n) = 
 \tx_n + \frac{a + \sqrt{(2a-1)^2 + 8a\tx_n}}
          2, \]
which yields by induction
\[  \tx_n = \frac{an^2 + (2a-1) \, n}{2} . \]

\begin{example}  The case $\kappa = 8/3$ gives the 
(chordal or half-plane)
Brownian intersection
exponents.  Here we have
\[  q(\lambda) = - \frac 14 + \frac 14 \,\sqrt{1 + 24\lambda},\]\[
           \tx(\beta,\lambda) =\beta + 
 \lambda +    \frac{q(\beta)}{3}
   \left[\sqrt{1 + 24 \lambda} - 1\right]. \]
In particular, $q(1) = 1 , q(2) = 3/2$ and hence
\[      \tx(1,\lambda) =   1 +
\lambda  + \frac 43 \, q(\lambda ) = \lambda  + \frac 23
  +\frac 13 \,   \sqrt{1 + 24 \lambda},
 \]
\[   \tx(2,\lambda) =  
       2 + \lambda + 2q(\lambda)  =   \lambda + \frac 32
   + \frac 12  \,   \sqrt{1 + 24 \lambda}. \]
\end{example}

\section{Nonintersecting $SLE_\kappa$}

We will assume $\kappa \leq 4$ and $k$ is a positive
integer.  Suppose $D$ is a simply connected domain and
$z_1,\ldots,z_k,w_1,\ldots,w_k$ are distinct points on
the boundary of $D$.  We will assume that $\p D$ is locally
analytic near all of these boundary points. We define
$SLE_D(z_1,\ldots,z_k;w_1,\ldots,w_k)$ which is a measure
on $k$-tuples of curves $\bar \gamma = (\gamma^1,\ldots,
\gamma^k)$ with $\gamma^j$ connecting $z_j$ to $w_j$. 
The measure is defined inductively on $k$.
\begin{itemize}
\item  If $k=1$, $\mu_D(z_1,w_1)$ is $SLE_\kappa$ from
$z_1$ to $w_1$ in $D$.
\item  Suppose $SLE_D(z_1,\ldots,z_{k-1}; w_1,\ldots,w_{k-1})$ 
has been defined.  This is a measure on $(k-1)$-tuples
$\hat \gamma = (\gamma^1,\ldots,\gamma^{k-1})$.  For each
realization of $\hat \gamma$, let $D_{\hat \gamma}$
denote the connected component of $D \setminus {\hat \gamma}$
that contains $z_k$ and $w_k$ on the boundary.  Then
for the marginal on $\hat \gamma$ we define
\[    \frac{d\mu_D(z_1,\ldots,z_k;w_1,\ldots,w_k)}
   {d\mu_D(z_1,\ldots,z_{k-1} ;w_1,\ldots,w_{k-1})} \, (\bar \gamma)
            = C(D_{\hat \gamma}; z_k, w_k ). \]
If $z_k,w_k$ are not on the boundaries of the same connected
component we set $ C(D_{\hat \gamma}; z_k, w_k ) = 0$.
\item  The conditional probability measure on
$z_k,w_k$ given $\hat \gamma$ is that of
$\mu_{D_{\hat \gamma}}(z_k,w_k)$.
\end{itemize}
While the definition above makes it appear that the definition
depends on the order that we specify the pairs $(z_1,w_1),\ldots,
(z_k,w_k)$, we can give an equivalent definition that is
independent.  Let $Y = Y_D(z_1,\ldots,z_k,w_1,\ldots,w_k)$ denote
the Radon-Nikodym derivative of $\mu_D(z_1,\ldots,z_k;w_1,\ldots,w_k)$ 
with respect to the product measure
\[   \mu_D(z_1,w_1) \times \cdots \times \mu_D(z_k,w_k) . \]
Then if $\bar \gamma = (\gamma^1,\ldots,\gamma^k)$,
\[    Y(\bar \gamma)= 1\{\gamma^j \cap \gamma^n = \emptyset,
              1 \leq j < n \leq k\} \, e^{-\Theta \c/2} \]
where $\Theta$ denotes the Brownian loop measure of the set of
loops that intersect at least two of the curves.  We will not
give the proof which is essentially an application of the
ideas from Lecture 4.

     The quantity
 \[   \frac{C(D;z_1,z_2,\ldots,z_k;w_1,w_2,\ldots,w_k)}
{C(D;z_1,w_1) \, C(D;z_2,w_2) \, \cdots \, C(D;z_k,w_k)}\]
is a conformal invariant. For $k=2$, it is given by $\phi(L)$ defined
in \eqref{may30.1}.

\section{Radial exponent and SAW III}

Let us consider radial $SLE_{\kappa}$ in
$\Disk$ from $1$ to $0$ parametrized
as in Lecture 5,
Section \ref{originalsec} with $g_t'(0) = e^{2at}$. Let $\theta
\in (0,\pi)$.  We will consider $\E[|g_t'(e^{2i\theta})|^\lambda].$
As seen in that section, this is an estimate about a one-dimesional
diffusion.  Indeed,
\[  
|g_t'(e^{2i\theta})|^\lambda = \exp\left\{-a\lambda \int_0^t
  \frac{dt}{\sin^2 \Psi_t} \, dt \right\},\]
where $\Psi_t$ satisfies
\begin{equation}  \label{may26.3}
       d\Psi_t = a \, \cot \Psi_t \, dt + dW_t. 
\end{equation}

\begin{proposition}  \label{prop.may23.rad}
If $a \geq 1/4$ and $\lambda > \lambda_0$, then
there exists $c$ such that 
\[   \E[|g_t'(e^{2i\theta})|^\lambda] \sim c \, e^{-2a\beta t}
  \, \sin^r \theta  = c \, g_t'(0)^{-\beta   } \, \sin^q
  \theta . \]
where 
\begin{equation}  \label{may26.1}
            q=q(\lambda)
  = \frac{(1-2a) + \sqrt{(2a-1)^2 + 8a\lambda}}{2},  
\;\;\;\;      \beta = \beta(\lambda)=
 \frac  \lambda  2 + \frac q{4a}. 
\end{equation}
\end{proposition}

\begin{proof} (Sketch)  If 
\begin{equation}  \label{may26.1.1}
        r^2 +r(2a-1) - 2a\lambda =0, \;\;\;\;
   k = a \lambda + \frac r2 ,
\end{equation}
then $M_t = |g_t'( e^{2i\theta})|^\lambda \,  \sin^r \Psi_t \,
  e^{k t} ,$ is a local martingale satisfying
\[               dM_t = r \, \cot \Psi_t \, M_t \, dW_t. \]
If $r$ is chosen as in \eqref{may26.1}, then using Girsanov we
can show that $M_t$ is actually a martingale and hence
\[           \sin^r \theta = e^{kt} \,
\E\left[ |g_t'( e^{2i\theta})|^\lambda \,  \sin^r \Psi_t 
   \right]. \]
We then proceed as in Proposition \ref{prop.may23}.
\end{proof}

\begin{example}  We have already seen an example of this
exponent.  Let $\lambda = b$.  Then $r = a, \beta = 3a/4$.  
If $C$ denotes the raidal $SLE$ total mass,
then
\[   C(\Disk \setminus \gamma(0,t];e^{2i\theta},0)
        = g_t'(0)^{\tilde b} \, |g_t'(e^{2i \theta})|^b . \]
Hence we can interpret this as
\[  \E\left[    C(\Disk \setminus \gamma(0,t];e^{2i\theta},0)
 \right] \sim c \, \sin^b \theta \, g_t'(0)^{\tilde b - \frac {3a}{4}}.\]

\end{example}

\begin{example} (SAW III)
 Here we will show how to predict a critical exponent
for self-avoiding walk.  Define the exponent $u$ by saying that
the probability that two self-avoiding walks of $n$ steps starting
at the origin avoid each other decays like $n^{-u}.$  
The typical distance of a SAW of length $n$ is $n^{1/d}$
where $d = 4/3$ is the fractal dimension of the paths.
 Hence, we can rephrase
the definition of $u $
 as saying that the probability that two SAWs go distance
$R$ without an intersection decays like $R^{-ud}$.  Given the 
conjectured relation between SAWs and $SLE_{8/3}$, we can guess
that if $\gamma^1,\gamma^2$ are  independent radial  $SLE_{8/3}$
in $\Disk$ going to $0$ starting at $+1,-1$, respectively, then
the probability that $\gamma^1$ reaches the ball of radius
$\delta$ about the origin without hitting the path of $\gamma^2$
decays like $\delta^{ud}$.  The restriction property tells us that
for a fixed realization of $\gamma^1$ (with corresponding
conformaal maps $g_t$), the probability that 
$\gamma^2$ avoids $\gamma^1(0,t]$ is given by
\[                g_t'(0)^{\tilde b} \, |g_t'(-1)|^b, \]
with $\tilde b = 5/48, b = 5/8$. Also, the Koebe-(1/4) theorem
tells us that the distance from the origin to $\gamma^1(0,t]$ is
comparable to $g_t'(0)$.  Our proposition tells us that
\[   \E[ |g_t'(-1)|^b] \asymp g_t'(0)^{-\beta}, \;\;\;\;
   \beta = \beta(b) = \frac{3a}{4} = \frac{9}{16}. \]
Hence the probability of no intersection is given by
\[ \E[ g_t'(0)^{\tilde b} \, |g_t'(-1)|^b] \sim g_t'(0)^{-11/24}. \]
This gives $ud = 11/24, u = 11/32$ and
matches the prediction first given by Nienhuis.  The exponent
$u$ is the same as $\gamma -1$ described in Lecture 1, Section \ref{sawsec}.
\end{example}

The radial exponent $\xi$ is defined by
\[       \xi(b, \lambda) = \tilde b + \beta(\lambda) 
= \tilde b + \frac \lambda 2 +
      \frac{q(\lambda)}{4a} . \]
More generally, we define
\[       \xi(b,\lambda_1,\ldots,\lambda_k) =
  \xi(b,\tx(\lambda_1,\ldots,\lambda_k)) . \]
Note that the chordal exponent $\tx$ appears in the definition
of $\xi$.  The exponent $\xi(\lambda_1,\lambda_2)$ is
defined so that the foll0wing holds:
\[  \xi(b,\lambda,\lambda_2) = \xi(b ,\tx(\lambda,\lambda_2))
       = \xi(\tx(b,\lambda_1),\lambda_2) . \]

\begin{example}  If $\kappa = 8/3$, then $\tx(5/8,1/8) = 1,
q(1/8) = 1/4,
\tx(5/8,5/8) = 2, q(5/8) = 3/4$, 
\[   \xi(5/8,\lambda) = 
\frac{5}{48} + \frac \lambda 2 +
      \frac 13 \, {q(\lambda)}  = \frac{1}{48} + \frac \lambda 2 +
   \frac 1{12} \, \sqrt{1 + 24 \lambda}
  . \] 
One can check that
\[  \xi(1,\lambda) = \frac 18 + \frac {\lambda}{2} +
  \frac 1 8\, \sqrt{1 + 24 \lambda}, \]
\[  \xi(2,\lambda) = \frac {11}{24} + \frac \lambda 2 + \frac{5}{
 24} \,  \sqrt{1 + 24 \lambda}. \]
These are the one-sided and two-sided Brownian intersection
exponents.  The values $\xi(1,0) = 1/4, \xi(2,0) = 2/3$
are the one-sided and two-sided disconnection exponents.
\end{example}

\section{Exercises for Lecture 6}

\begin{exercise}
Find the constant $c$ in Proposition \ref{prop.may23}.
\end{exercise}

\begin{exercise}
Complete the details in the proof of Proposition \ref{prop.may23.rad}.
Note that in the proposition,  $r$ was chosen to be a particular root
of the quadratic equation in \eqref{may26.1.1}.  Show that if
we choose $r $ to be the other root of that equation, then
$M_t$ is not a martingale (even though it is a local martingale).
Express $c$ in terms of invariant densities for diffusions of
the form \eqref{may26.3}.
\end{exercise}

\begin{exercise}  In Proposition \ref{prop.may23.rad}, note that
$\beta > \lambda/2$ if $\lambda > 0$ and 
$\beta \sim \lambda/2$  as $\lambda \rightarrow \infty$.  
Show how these facts can be deduced from   the
Beurling estimate.
\end{exercise}

\begin{exercise} Fix $\kappa \leq 4$ and positive integer $n$.
Let $0 < y_1 < \cdots < y_n < \pi$ and let $\rect_L
= [0,L] \times [0,\pi].$ Let $z_j = y_ji, w_j = w_{j,L} =
 L +  y_ji$.  Show that as $L \rightarrow \infty$,
\[  C(\rect_L;z_1,\ldots,z_n;w_1,\ldots,w_n) \asymp
              e^{-L \tx_n} , \]
where $\tx_n$ is as defined in Section \ref{intersec}. 
\end{exercise}

\begin{exercise}  For $\kappa =2$, it can be shown using
an idenity of Fomin that
\[ C(\rect_L;z_1,\ldots,z_n;w_1,\ldots,w_n) = 
   \det\left[H_{\p \rect_L}(z_j,w_k)\right]. \]
Use this identity to give another proof that 
\[            \tx_n = \frac{n^2+n}{2}, \;\;\;\;
  \kappa =2. \]
\end{exercise}

\lecture*{Tables for reference}

\begin{table}\caption{Parameters for $SLE$}
\begin{tabular}{|c|c|c|}   \hline   
\raisebox{-1.5ex}{$\kappa$} & \raisebox{-1.5ex} {variance
 in driving function}   & \raisebox{-1.5ex}  {$\kappa > 0$}
\\[3ex] \hline   
          \raisebox{-1.5ex}{  $a$}
 & \raisebox{-1.5ex} {rate of capacity growth}
 & \raisebox{-1.5ex} {$\frac 2 \kappa$} \\ [3ex] \hline   
            \raisebox{-1.5ex}{$d$} & \raisebox{-1.5ex} 
{dimension of paths} &  \raisebox{-1.5ex}{$1 + \frac \kappa 8 =
                1 + \frac 4 a, \;\;\;\;  \kappa \leq 8$} \\  [3ex] \hline
            \raisebox{-1.5ex}{$b$}
 &  \raisebox{-1.5ex}{boundary scaling exponent}
   &   \raisebox{-1.5ex}{$\frac{3a-1}{2} =
                    \frac{6 - \kappa}{2\kappa}$ }\\ [3ex]\hline
           \raisebox{-1.5ex}{ $\tilde b$ }
&  \raisebox{-1.5ex}{one-sided interior scaling exponent}
   &  \raisebox{-1.5ex}{$\frac{1-a}{2a} \,b =
  \frac{\kappa -2}{4} \, b$} \\[3ex] \hline
   \raisebox{-1.5ex}{ $\hat b$ }
&  \raisebox{-1.5ex}{two-sided interior scaling exponent}
   &  \raisebox{-1.5ex}{$2-d =
   1 - \frac{\kappa}{8}$} \\[3ex] \hline  \raisebox{-1.5ex}{$ \c$ }
 &  \raisebox{-1.5ex}{central charge }
&  \raisebox{-1.5ex}{$\frac{2b(3-4a) }{a} = \frac{(3-4a)(3a-1)}
  a = \frac{(3 \kappa - 8)(6 - \kappa)}{2 \kappa}$} \\[3ex] \hline
\end{tabular}
\end{table}

\begin{table}\caption{Some discrete models}

\begin{tabular}{|c|c|c|c|c|c|c|}  \hline
  & \raisebox{-1.5ex}{$\kappa$}
& 
\raisebox{-1.5ex}{$a$} &\raisebox{-1.5ex}{$b$} &
\raisebox{-1.5ex}{$\tilde b$} & 
\raisebox{-1.5ex}{$\c$}& \raisebox{-1.5ex}{$d$}\\[3ex]
 \hline  
\raisebox{-1.5ex}{loop-erased walk} & \raisebox{-1.5ex}2 &  \raisebox{-1.5ex}
   1 &
\raisebox{-1.5ex}1 & \raisebox{-1.5ex}{$0$} & \raisebox{-1.5ex}{$-2$} & 
  \raisebox{-1.5ex}{$\frac 54$}\\[3ex]
\hline
\raisebox{-1.5ex}{self-avoiding walk} & \raisebox{-1.5ex}{$\frac 83$} &  \raisebox{-1.5ex}
   {$\frac 34$}  &
\raisebox{-1.5ex}{$\frac 58$} & \raisebox{-1.5ex}{$\frac 5{48}$} &\raisebox{-1.5ex}{$0$} & 
  \raisebox{-1.5ex}{$\frac 43$}\\[3ex]
\hline
\raisebox{-1.5ex}{Ising interface} 
 & \raisebox{-1.5ex}{$3$} &  \raisebox{-1.5ex}
   {$\frac 23$}  &
\raisebox{-1.5ex}{$\frac 12$} & \raisebox{-1.5ex}{$\frac 18 $}
    & \raisebox{-1.5ex}{$\frac 12 $} & 
  \raisebox{-1.5ex}{$\frac{11}8$}\\[3ex]
\hline
\raisebox{-1.5ex}{harmonic explorer}&&&&& &\\
  \raisebox{-1.5ex}{free field interface}
 & {$4$} &  
   {$\frac 12$}  &
 {$\frac 14 $} & {$\frac 18 $}
    &  {$1 $} & 
 {$\frac 32$}\\[3ex]\hline
\raisebox{-1.5ex}{percolation interface} 
 & \raisebox{-1.5ex}{$6$} &  \raisebox{-1.5ex}
   {$\frac 13$}  &
\raisebox{-1.5ex}{$0$} & \raisebox{-1.5ex}{$0 $}
    & \raisebox{-1.5ex}{$0 $} & 
  \raisebox{-1.5ex}{$\frac 74$}\\[3ex]\hline
\raisebox{-1.5ex}{uniform spanning tree} 
 & \raisebox{-1.5ex}{$8$} &  \raisebox{-1.5ex}
   {$\frac14$}  &
\raisebox{-1.5ex}{$- \frac 18 $} & \raisebox{-1.5ex}{$  -\frac 3{16}$}
    & \raisebox{-1.5ex}{$-2 $} & 
  \raisebox{-1.5ex}{$2$}\\[3ex]\hline
\end{tabular}
\end{table}
 









\end{document}